\documentclass[a4paper,10pt]{amsart}
\usepackage{epsfig}
\usepackage{graphicx,color}
\usepackage{amsmath,amsfonts,amssymb,eufrak}
\usepackage{float}
\usepackage{graphics}
\usepackage{latexsym}
\usepackage{float}
\usepackage{verbatim}
\usepackage{xy}
\xyoption{matrix}
\xyoption{arrow}
\input xy
\usepackage{rotating}
\usepackage{caption}

\xyoption{all}         

\newtheorem{lem}{Lemma}[section]
\newtheorem{prop}[lem]{Proposition}
\newtheorem{cor}[lem]{Corollary}
\newtheorem{thm}[lem]{Theorem}

\newcommand{\Ext}{\operatorname{Ext}\nolimits}

\newcommand{\mo}{\operatorname{mod}\nolimits}

\newcommand{\End}{\operatorname{End}\nolimits}
\newcommand{\add}{\operatorname{add}\nolimits}

\newcommand{\id}{\operatorname{id}\nolimits}

\begin{document}
\title[A geometric realization...]{A geometric realization of the
  $m$-cluster category of type $\widetilde{A}$}
\author{Hermund Andr\' e Torkildsen}

\begin{abstract}
We give a geometric realization of a subcategory of the $m$-cluster
category $\mathcal{C}^m$ of type $\widetilde{A}_{p,q}$, by using
$(m+2)$-angulations of an annulus with $p+q$ marked points. We also
give a bijection between an equivalence class of $(m+2)$-angulations
and the mutation class of coloured quivers of type
$\widetilde{A}_{p,q}$. %For $m=1$ we classify all endofunctors $F$ on
%$\mathcal{C}$ such that for a cluster-tilting object $T$ we have
%$\End(T) \simeq \End(FT)$.  
\end{abstract}

\maketitle

\section*{Introduction}
The cluster category was defined in \cite{bmrrt} in general and in
\cite{ccs} for Dynkin Type $A$. Their motivation was to categorify the
combinatorics of cluster algebras defined by Fomin and Zelevinsky
\cite{fz1}. In \cite{ccs} the authors defined a category with
indecomposable objects the diagonals in a regular polygon. The
geometric model was later extended to Dynkin type $D$ by Schiffler
\cite{s} and to Dynkin type $\widetilde{A}$ by Br\"ustle and Zhang
\cite{bz}.  

The $m$-cluster category (see for example \cite{k,t,w,z,zz})
generalizes the cluster category. Baur and Marsh extended the geometric
models of the cluster category of type $A$ \cite{bm2} and $D$
\cite{bm1} to $m$-cluster categories. For example, in the $A$-case,
they consider $m$-diagonals of regular polygons. In this paper we will
consider $m$-cluster categories of type $\widetilde{A}$. 

The results in this paper are a part of the author's PhD thesis from
December 2010 \cite{to4}.

\section{$m$-cluster categories}

In \cite{bmrrt} the cluster category was defined as an orbit category
of the derived category. Let $H=kQ$ be a finite dimensional hereditary
algebra over an algebraically closed field $k$, where $Q$ is a
quiver. The cluster category $\mathcal{C}_H$ is the orbit category
$\mathcal{D}^b(H)/\tau^{-1}[1]$, where $\tau$ is the Auslander-Reiten
translation, $[1]$ is the shift functor and $\mathcal{D}^b(H)$ the 
bounded derived category of $\mo H$. We can also consider the orbit
category $\mathcal{D}^b(H)/\tau^{-1}[m]$, and this is the
$m$-cluster category. The $m$-cluster category is a Krull-Schmidt
category for all $m$, and it has an AR-translate $\tau$. From \cite{k}
we also know that it is a triangulated category for all $m$. %The
%indecomposable objects in $\mathcal{C}_H^m$ are of the form $X[i]$,
%with $0 \leq i < m$, where $X$ is an indecomposable $H$-module, or of
%the form $P[m]$, where $P$ is a projective $H$-module. 

The $m$-cluster categories come equipped with a class of objects
called $m$-cluster tilting objects. An $m$-cluster tilting object $T$
is an object with the property that $X$ is in $\add T$ if and only if 
$\Ext_{\mathcal{C}_H^m}^i(T,X)=0$ for all $i \in \{1,2,...,m\}$. An
object $X$ is called maximal $m$-rigid if it has the property that $X
\in \add T$ if and only if $\Ext_{\mathcal{C}_H^m}^i(T \oplus X,T
\oplus X)=0$ for all $i \in \{1,2,...,m\}$. A maximal $m$-rigid object 
is an $m$-cluster tilting object \cite{w,zz}, and an $m$-cluster
tilting object has always $n$ non-isomorphic indecomposable summands
\cite{z}. The algebra $\End_{\mathcal{C}_H^m}(T)$ is called an
$m$-cluster tilted algebra when $T$ is $m$-cluster tilting.

If $\bar{T}$ is an object in $\mathcal{C}_H^m$ with $n-1$
non-isomorphic indecomposable direct summands, such that
$\Ext_{\mathcal{C}_H^m}^i(\bar{T},\bar{T})=0$ for $i \in
\{1,2,...,m\}$, there exist exactly $m+1$ non-isomorphic objects $T'$
(called complements) such that $\bar{T} \oplus T'$ is an $m$-cluster
tilting object \cite{w,zz}. The object $\bar{T}$ is called an almost
complete $m$-cluster tilting object. Let $T_k^{(c)}$, where $c\in
\{0,1,2,...,m\}$, be the complements of $\bar{T} = T/T_k$. Then we
know from \cite{iy} that the complements are connected by $m+1$
exchange triangles $$T_k^{(c)} \rightarrow B_k^{(c)} \rightarrow
T_k^{(c+1)}\rightarrow,$$ where $B_k^{(c)}$ is in $\add{\bar{T}}$.

\section{Quiver mutation}

Quiver mutation was defined by Fomin and Zelevinsky in their work with
cluster algebras. Buan and Thomas extended quiver mutation to a class
of coloured quivers to model mutation in $m$-cluster categories. Let
$T$ be an $m$-cluster tilting object. In \cite{bt} they associate to $T$ a
coloured quiver $Q_T$ in the following way. There is a vertex in $Q_T$
for every indecomposable summand of $T$. The arrows have colours
chosen from the set $\{0,1,2,...,m\}$. If $T_i$ and $T_j$ are two
indecomposable summands of $T$ corresponding to vertex $i$ and $j$ in
$Q_T$, there are $r$ arrows from $i$ to $j$ of colour $c$, where $r$
is the multiplicity of $T_j$ in $B_i^{(c)}$.

They show that quivers obtained in this way have no loops. Also, if
there is an arrow from $i$ to $j$ with colour $c$, then there is no
arrow from $i$ to $j$ with colour $c' \neq c$. If there are $r$ arrows
from $i$ to $j$ of colour $c$, then there are $r$ arrows from $j$ to
$i$ of colour $m-c$. 

Coloured quiver mutation keeps track of the exchange of indecomposable
summands of $m$-cluster tilting objects. The mutation of $Q_T$ at
vertex $j$ is defined as the quiver $\mu_j(Q_T)$ obtained as follows.  

\begin{enumerate}
\item For each pair of
  arrows $\xymatrix{i\ar[r]^{(c)}&j\ar[r]^{(0)}&k\\}$, where $i \neq
  k$ and $c \in \{0,1,...,m\}$, add an arrow from $i$ to $k$ of colour
  $c$ and an arrow from $k$ to $i$ of colour $m-c$. 
\item If there exist arrows of different colours from a vertex $i$ to
  a vertex $k$, cancel the same number of arrows of each colour until
  there are only arrows of the same colour from $i$ to $k$. 
\item Add one to the colour of all arrows that goes into $j$, and
  subtract one from the colour of all arrows going out of $j$.
\end{enumerate}

In \cite{bt} they prove that if $T = \oplus_{i=1}^{n} T_i$ is an
$m$-cluster tilting object in $\mathcal{C}_H^m$ and $T' = T / T_j
\oplus T_j^{(1)}$ is an $m$-cluster tilting object where there is an
exchange triangle $T_j \rightarrow B_j^{(0)} \rightarrow T_j^{(1)}
\rightarrow$, then $Q_{T'} = \mu_j(Q_T)$. We note that this was already
known for $m=1$ \cite{bmr2}. In \cite{zz} it was shown that any
$m$-cluster tilting object can be reached from any other $m$-cluster
tilting object via iterated mutation. And in \cite{bt} the authors
show that for an $m$-cluster category $\mathcal{C}_H^m$, where $H=kQ$,
all quivers of $m$-cluster tilted algebras are given by repeated
mutation of $Q$. 

When $m=1$, it was shown in \cite{br} that the mutation class of an
acyclic quiver $Q$ is finite if and only if the underlying graph of
$Q$ is either Dynkin, extended Dynkin or $Q$ has at most two
vertices. This was generalized in \cite{to2}. A coloured quiver $Q$
corresponding to an $m$-cluster tilting object, has finite mutation
class if and only if $Q$ is mutation equivalent to a quiver $Q'$,
where the quiver obtained from $Q'$ by removing all arrows with colour
$\neq 0$ has underlying graph Dynkin or extended Dynkin, or it has at
most two vertices, and there are only arrows of colour $0$ and $m$ in
$Q'$. 

\section{Geometric descriptions of $m$-cluster categories}

We will not go into detail about the various geometric descriptions of
$m$-cluster categories, since we will do it in detail for $m$-cluster
categories of type $\widetilde{A}$. We refer to the papers
\cite{ccs,s,bm1,bm2,bz}. 

In \cite{ccs} the authors defined the cluster category of type $A_n$
by using regular polygons and diagonals between vertices on the border
of the polygons. In \cite{bm2} they generalized this to $m$-cluster
categories. Baur and Marsh considered an $(nm+2)$-gon 
$P_{nm+2}$ and $m$-diagonals between vertices on the border of
$P_{nm+2}$. An $m$-diagonal is a diagonal that divides
$P_{nm+2}$ into two parts with number of vertices congruent to $2$
modulo $m$. They defined an additive category where the indecomposable
objects are the $m$-diagonals. The morphisms are spanned by certain
elementary moves. The authors showed that this category is equivalent
to the $m$-cluster category of type $A_n$. A set of $m$-diagonals that divides
$P_{nm+2}$ into $(m+2)$-gons is called an $(m+2)$-angulation, and such
a set always has $n$ elements. We have that an $(m+2)$-angulation
$\Delta$ is an $m$-cluster tilting object in this category, and we can
define mutation on $m$-diagonals in $\Delta$. Also, given an
$(m+2)$-angulation $\Delta$, we can define a coloured quiver
$Q_{\Delta}$ (see \cite{bm2, bt}), and mutation on $\Delta$ commutes
with mutation on $Q_{\Delta}$. We mention that the Dynkin case
$\widetilde{A}$ has been considered in \cite{bz} for $m=1$. 

Given an $(m+2)$-angulation $\Delta$, there exist, as we mentioned
above, a coloured quiver $Q_{\Delta}$. In \cite{to1} it was shown that
there exist a bijection between the set of triangulations, where two
triangulations are equivalent if they are rotations of eachother, and
the mutation class of quivers of Dynkin type $A$. This was generalized
in \cite{to3} to $m$-coloured quivers and $(m+2)$-angulations of
$P_{nm+2}$. A similar result was obtained for Dynkin type $D$ in
\cite{bto}. In the second part of this paper, we will obtain a similar
result for Dynkin type $\widetilde{A}$.

\section{$(m+2)$-angulations of $P_{p,q,m}$} \label{triangulation}

Let $m \geq 1$, $p \geq2$ and $q \geq 2$ be integers, and set
$n=p+q$. Let $P_{p,q,m}$ be a regular $mp$-gon, with a regular
$mq$-gon at its center, cutting a hole in the interior of the outer
polygon. When $m=1$ we just write $P_{p,q}$. Denote by $P_{p,q,m}^0$
the interior between the outer and inner polygon. Label the vertices
on the outer polygon $O_0,O_1,...,O_{mp-1}$ in the counterclockwise
direction, and label the vertices of the inner polygon $I_0,
I_1,...,I_{mq-1}$, in the clockwise direction. See Figure
\ref{figannulusmarked}. If one of the polygons has $2$ vertices, we
draw the polygon as a circle with two marked points. For simplicity,
we always draw the polygons such that the vertices $O_0$ and $I_0$ are
as close as possible.

 \begin{figure}[htp]
  \begin{center}
    \includegraphics[width=5cm]{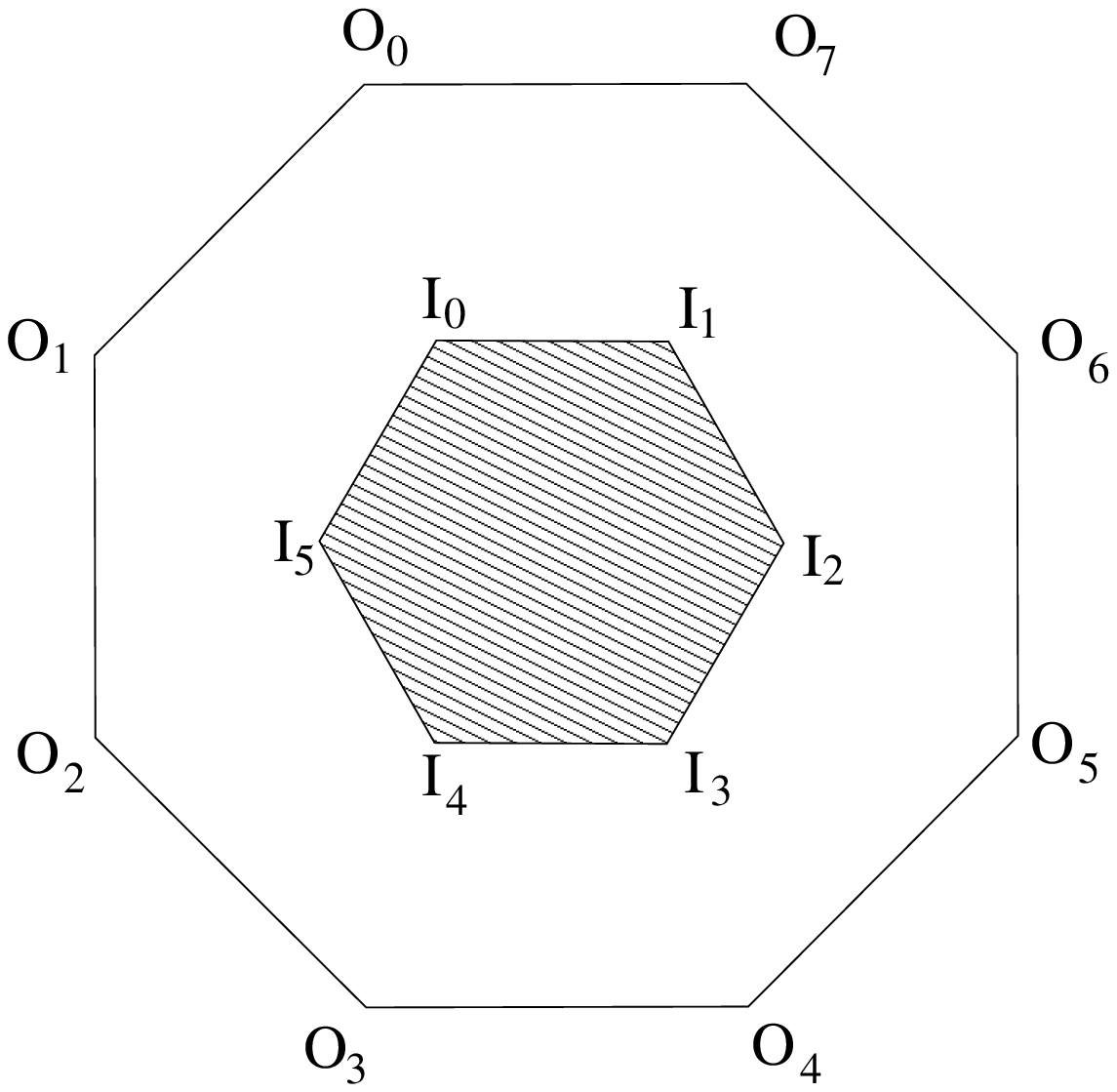}
    \includegraphics[width=5cm]{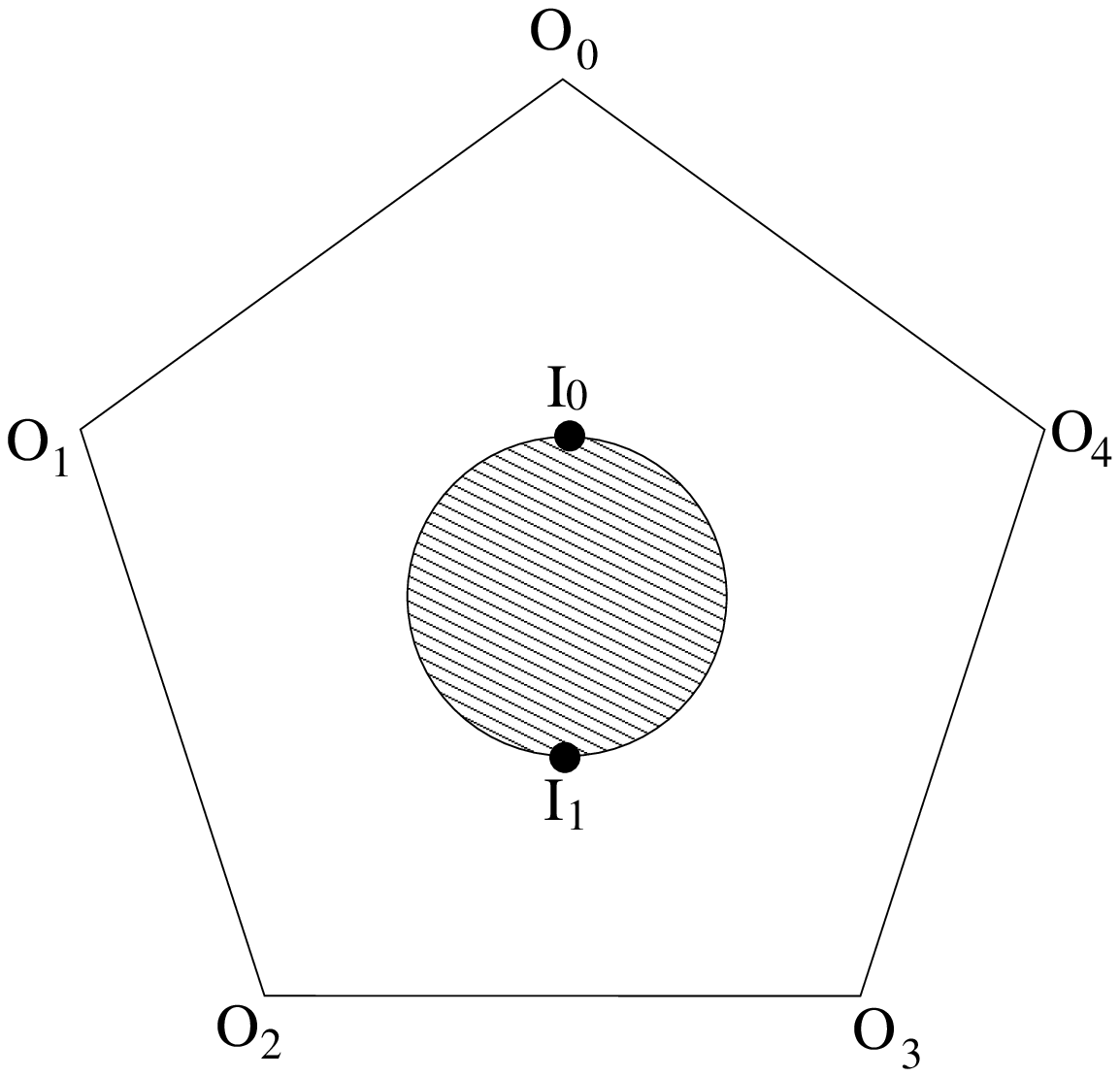}
  \end{center}\caption{\label{figannulusmarked} $P_{8,6}$ and
    $P_{5,2}$ for $m=1$.}
  \end{figure}

Let $\delta_{i,k}$ be the path in the counterclockwise direction from
$O_i$ to $O_{i+k-1}$ along the border of the outer polygon, where $k$
is the number of vertices that $\delta_{i,k}$ runs through (including
the start and end vertex). If $k = pm+1$, the path runs around the
polygon exactly once. If $k > pm + 1$, the path crosses
itself. Similarly, we denote by $\gamma_{i,k}$ the path in the
clockwise direction from $I_i$ to $I_{i+k-1}$, along the border of the
inner polygon, where $k$ is the number of vertices the path runs
through. Of course, here we always compute modulo $pm$ and $qm$. 

Now we consider paths of the following types.

\begin{itemize}
\item Type 1: A path in $P_{p,q,m}^0$ between a vertex on the outer
  polygon and a vertex on the inner polygon, i.e. a path between $O_i$
  and $I_j$ for some $i$ and $j$. %Such a path should not cross itself.
\item Type 2: A path $\alpha$ from $O_i$ to $O_{i+k-1}$ in
  $P_{p,q,m}^0$, such that $\alpha$ is homotopic to $\delta_{i,k}$ for
  some $k \geq 3$.%, and $\alpha$ crosses itself a minimal number of
  %times. 
\item Type 3: A path $\alpha$ from $I_i$ to $I_{i+k-1}$ in
  $P_{p,q,m}^0$, such that $\alpha$ is homotopic to $\gamma_{i,k}$ for
  some $k \geq 3$.%, and $\alpha$ crosses itself a minimal number of
  %times. 
\end{itemize}

 \begin{figure}[htp]
  \begin{center}
    \includegraphics[width=3cm]{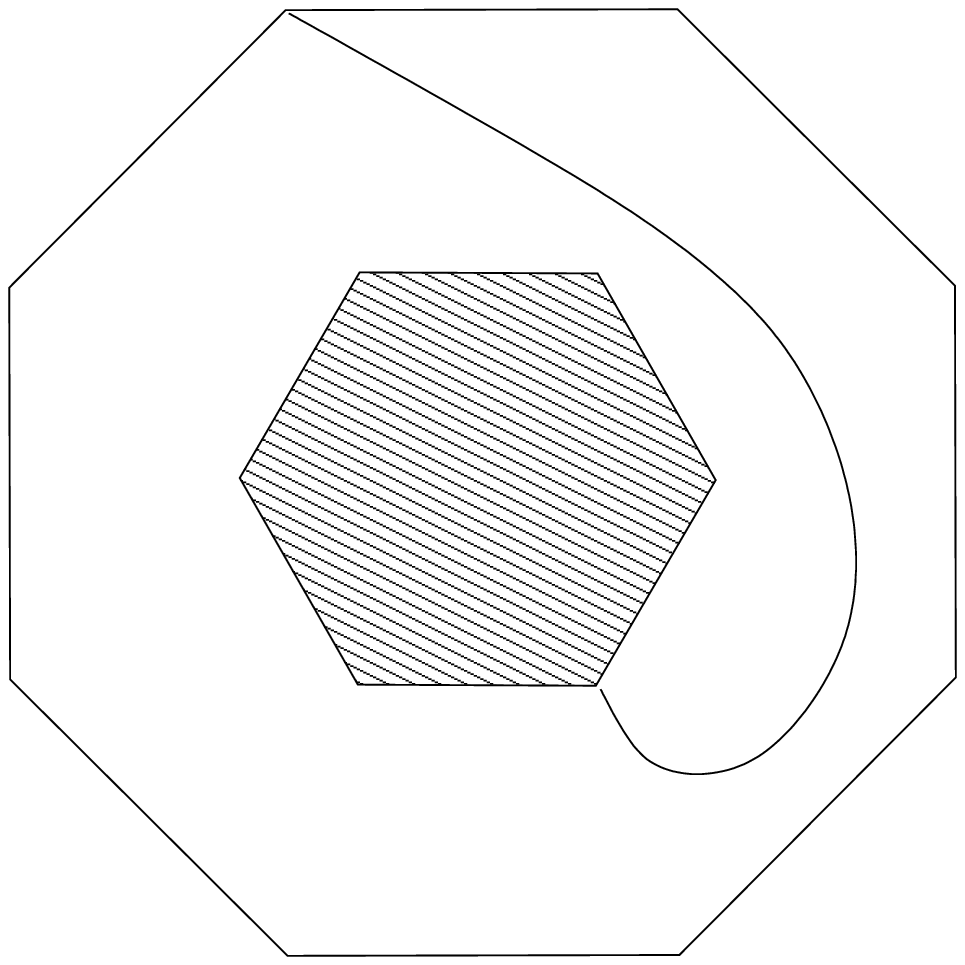}
    \includegraphics[width=3cm]{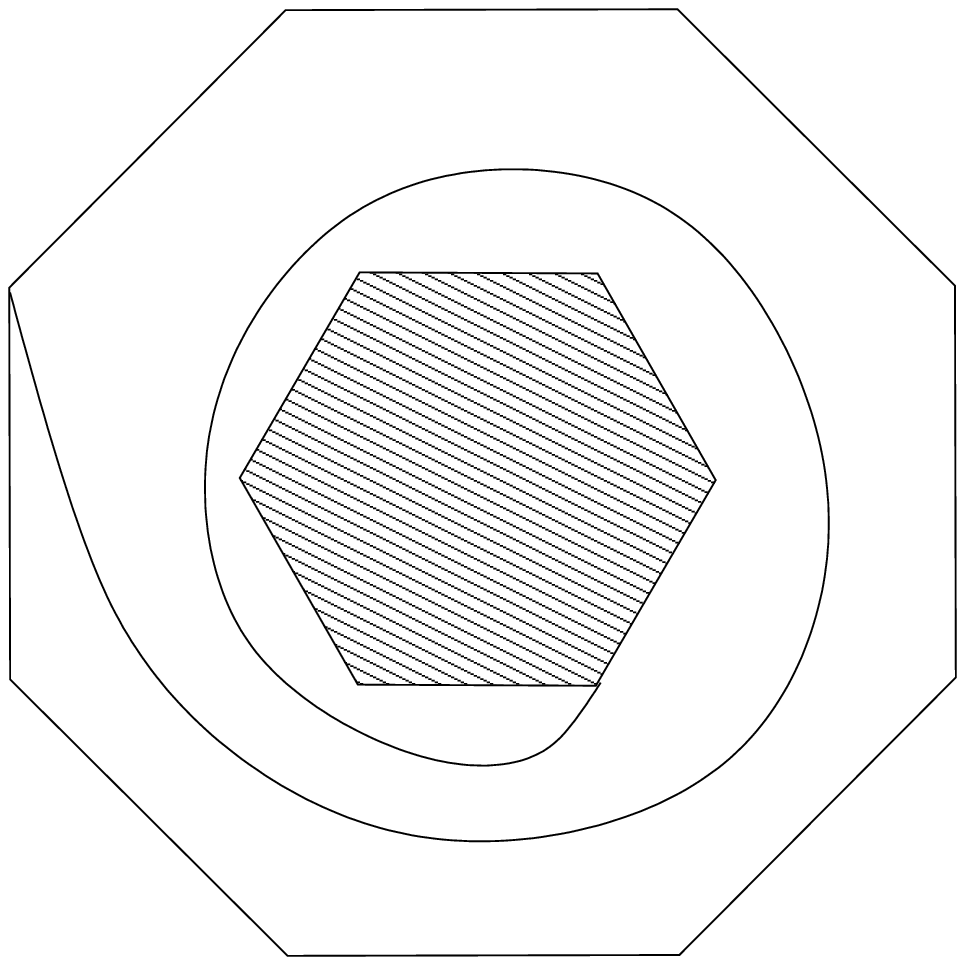}
    \includegraphics[width=3cm]{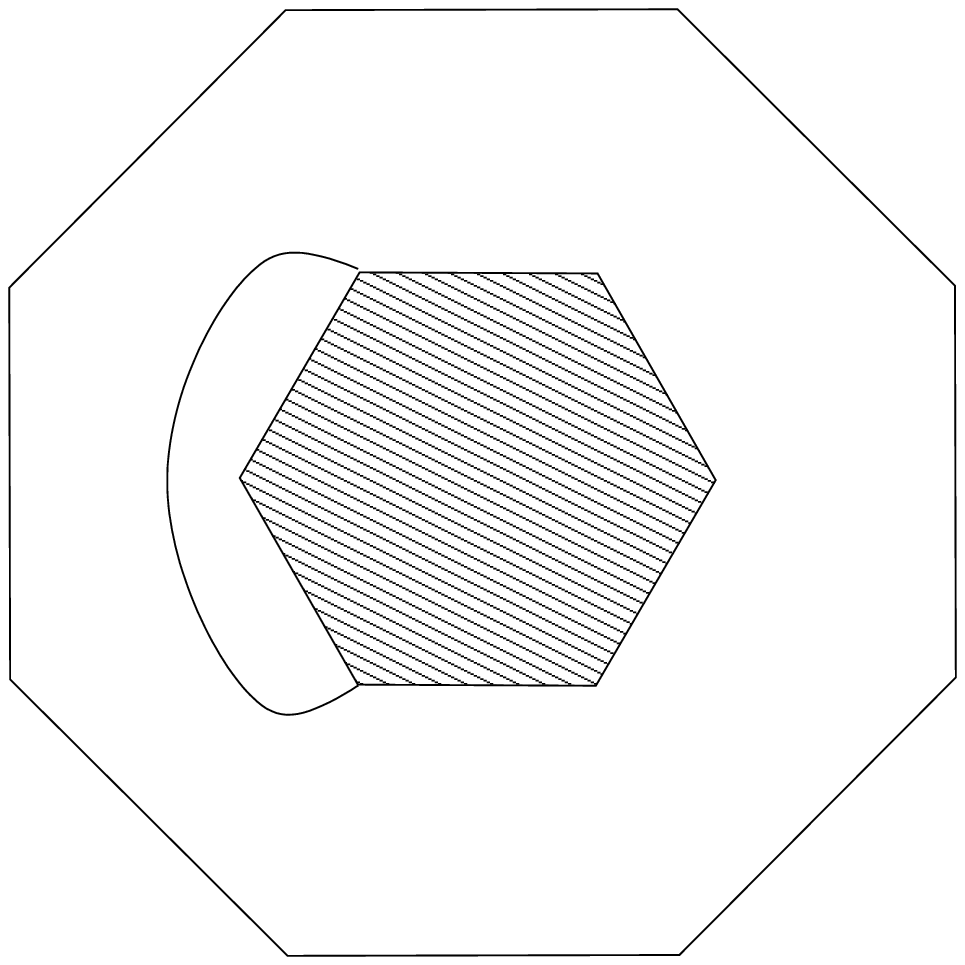}
    \includegraphics[width=3cm]{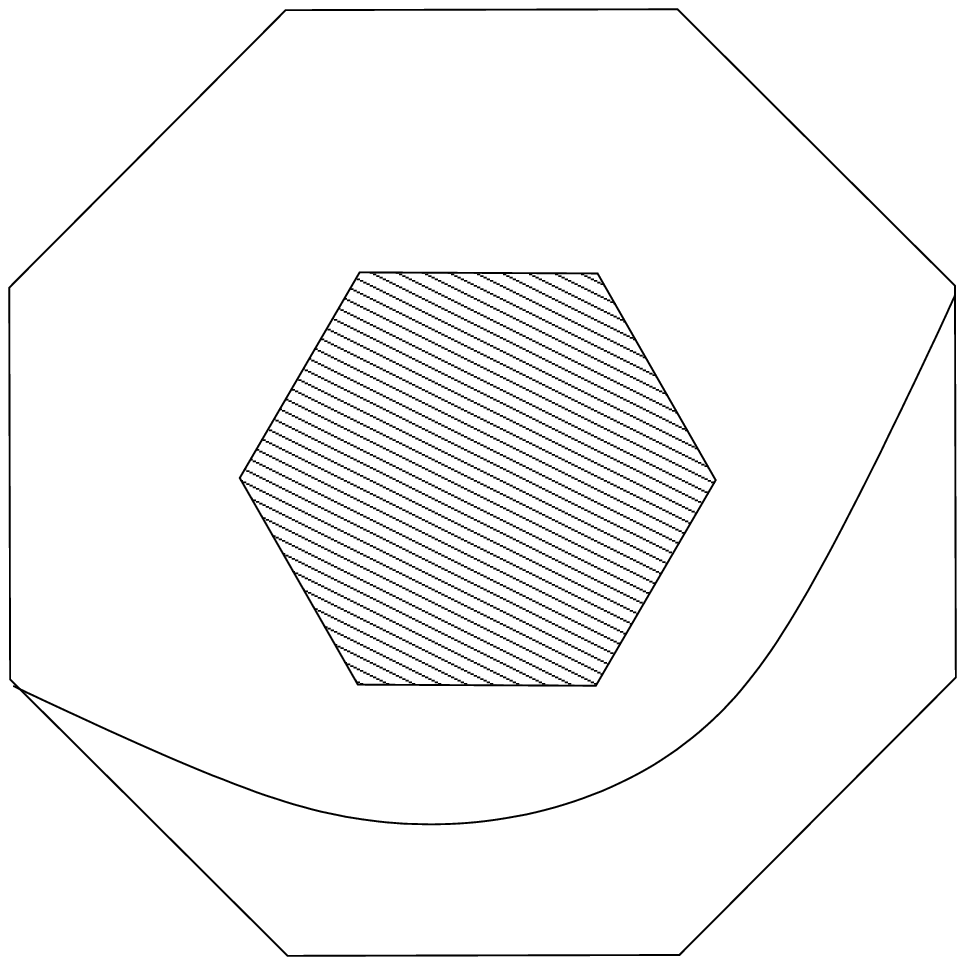}
  \end{center}\caption{\label{figannulusdiagonal} Some examples of
    diagonals in $P_{8,6}$.}
  \end{figure}

See Figure \ref{figannulusdiagonal} for some examples. Note that the
winding number of a path of any type can be greater than 1. Also, a
path of Type 1 is homotopic to a path that never crosses itself. 

Now, we define two paths to be equivalent if they start in the same
vertex, end in the same vertex and they are homotopic. We call these 
equivalence classes \emph{diagonals} in $P_{p,q,m}$. Let $O_{i,k}$
denote the diagonals homotopic to $\delta_{i,k}$, and let $I_{i,k}$ be
the diagonals homotopic to $\gamma_{i,k}$.  
%For a path of Type 1 we denote the equivalence class
%of diagonals between $O_{i}$ and $I_{j}$ with winding number $k$ by
%$U_{i,j,k}$. 

We define the \emph{crossing number} of any two diagonals, $\alpha$
and $\beta$, to be
 $$e(\alpha,\beta) = \min\{|\alpha \cap \beta \cap
P_{p,q,m}^0|\}.$$ 
The crossing number of a diagonal $\alpha$ is defined as the minimal
number of times the diagonal crosses itself in the interior. The
crossing number of any diagonal of Type 1 is always 0.

We say that two diagonals cross if the crossing number is not 0, and
we say that a diagonal crosses itself if the crossing number of the
diagonal is not 0. This enables us to define a triangulation, which is a
maximal set of diagonals which do not cross. See Figure
\ref{figannulustriangulation} for some examples of triangulations.

  \begin{figure}[htp]
  \begin{center}
    \includegraphics[width=4.6cm]{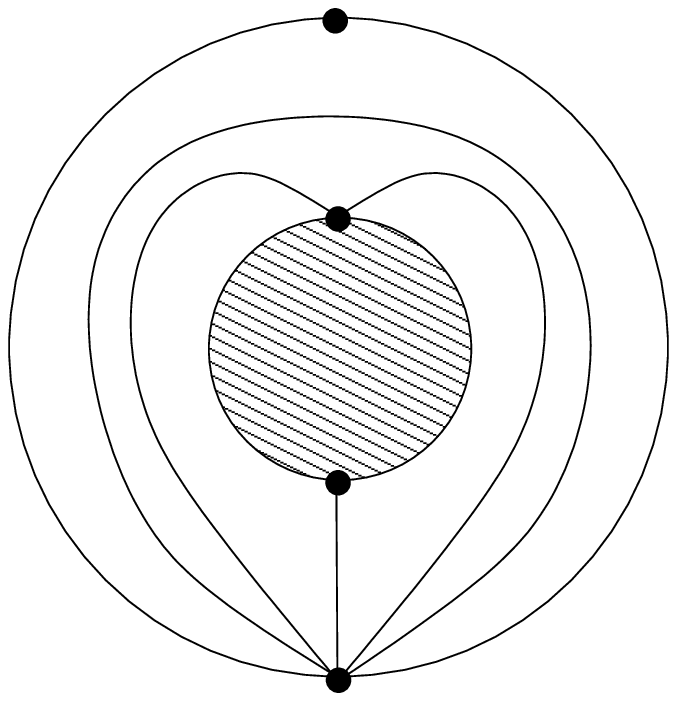}
    \includegraphics[width=4.6cm]{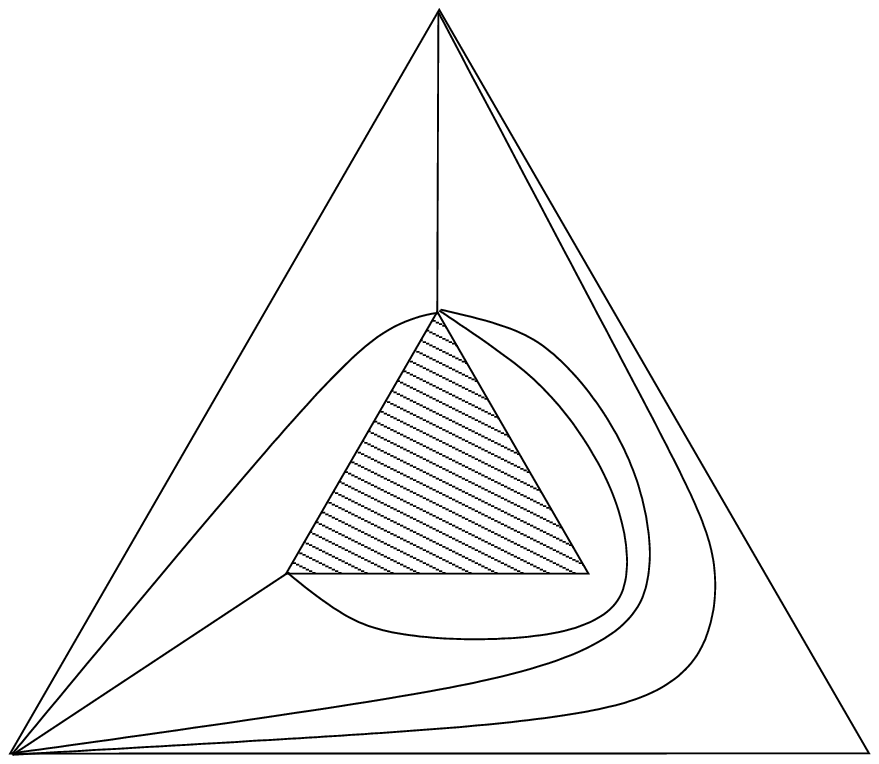}
  \end{center}\caption{\label{figannulustriangulation} Examples of
    triangulations of $P_{2,2}$ and $P_{3,3}$.}
  \end{figure}

We have the following easy lemma
%, which we will see later is
%equivalent to the fact that any cluster-tilting object has at least
%one indecomposable direct summand which is not regular.

\begin{lem}\label{atleastonediagonal}
In any triangulation of $P_{p,q,m}$, there exist at least one diagonal
of Type 1.
\end{lem}
%\begin{proof}
%Maa bevises.
%\end{proof}

Using the lemma, we can prove that the number of diagonals in any
triangulation is given by $p+q$. More generally, see \cite{fst} for an
annulus of $n_1+n_2$ marked points.  

\begin{prop}\label{numberofdiagonals}
Any triangulation of $P_{p,q,m}$ consists of exactly $p+q$ diagonals.
\end{prop}
\begin{proof}
Let $\Delta$ be a triangulation, and let $\alpha$ be a diagonal of
Type 1, which exist by Lemma \ref{atleastonediagonal}, say from $O_i$
to $I_j$. See Figure \ref{figcutting}, where we cut the polygon along
$\alpha$ and fold it out.
  \begin{figure}[htp]
  \begin{center}
    \includegraphics[width=12.5cm]{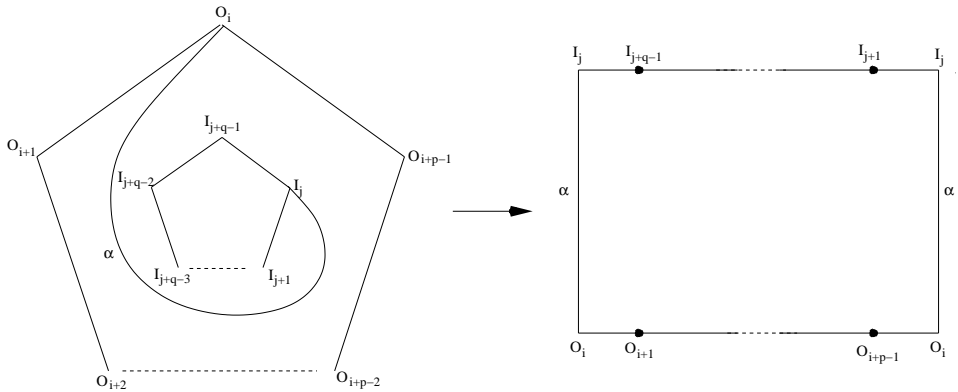}
  \end{center}\caption{\label{figcutting} See proof of Lemma
    \ref{numberofdiagonals}.} 
  \end{figure}
We obtain a polygon with $p+q+2$ vertices, and such a polygon can be
triangulated by $p+q-1$ diagonals by \cite{fst,ccs}. Then $\Delta$ has
$p+q$ diagonals, counting the diagonal $\alpha$.
\end{proof}

An $m$-diagonal in $P_{p,q,m}$ is a diagonal of the types above, but
with the following restrictions:

\begin{itemize}
\item The $m$-diagonals of Type 2 (Type 3) are of the form
  $O_{i,km+2}$ ($I_{i,km+2}$), where $k \geq 1$, for all $i$.
%Any $m$-diagonal of Type 2
%(3) divides the outer (inner) 
%polygon into two regions, each with number of vertices congruent to
%$2$ modulo $m$.
\item If $\alpha$ is an $m$-diagonal of Type 1 between $O_i$ and
  $I_j$,  then $i$ is congruent to $j$ modulo $m$.
\end{itemize}

We say that a set of $m$-diagonals cross if they intersect in the
interior $P^0_{p,q,m}$ (i.e. their crossing number as diagonals is not
$0$). A set of non-crossing $m$-diagonals that divides $P_{p,q,m}$
into $(m+2)$-gons is called an $(m+2)$-angulation. When $m = 1$ this
is a triangulation as described above.

We also have the following. 

\begin{prop}\label{numberofdiagonals2}
Any $(m+2)$-angulation of $P_{p,q,m}$ consists of exactly $p+q$
$m$-diagonals, and there exist at least one $m$-diagonal of Type 1.  
\end{prop}

See examples of $(m+2)$-angulations in Figure \ref{figexang}.

  \begin{figure}[htp]
  \begin{center}
    \includegraphics[width=6cm]{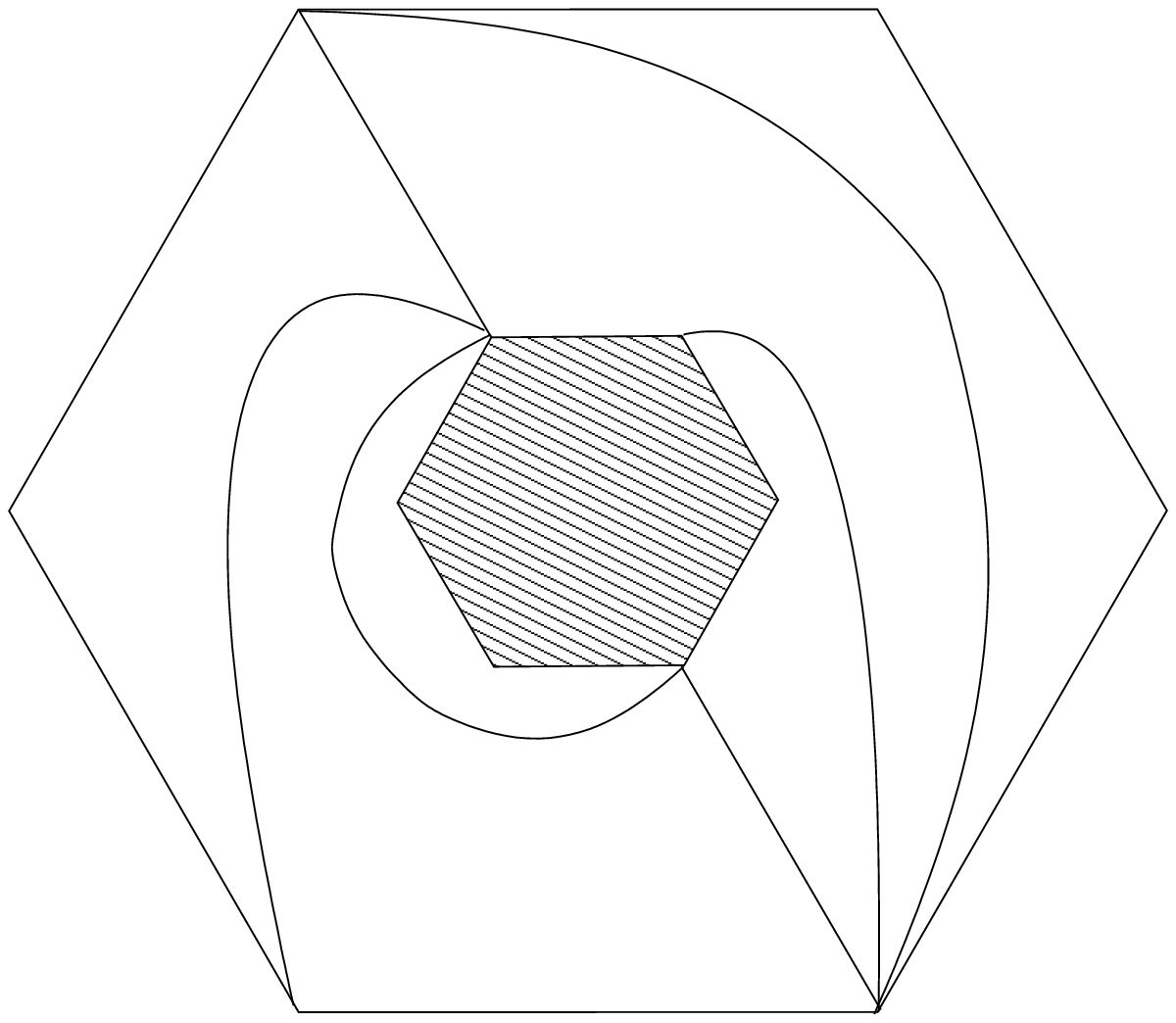}
    \includegraphics[width=5.3cm]{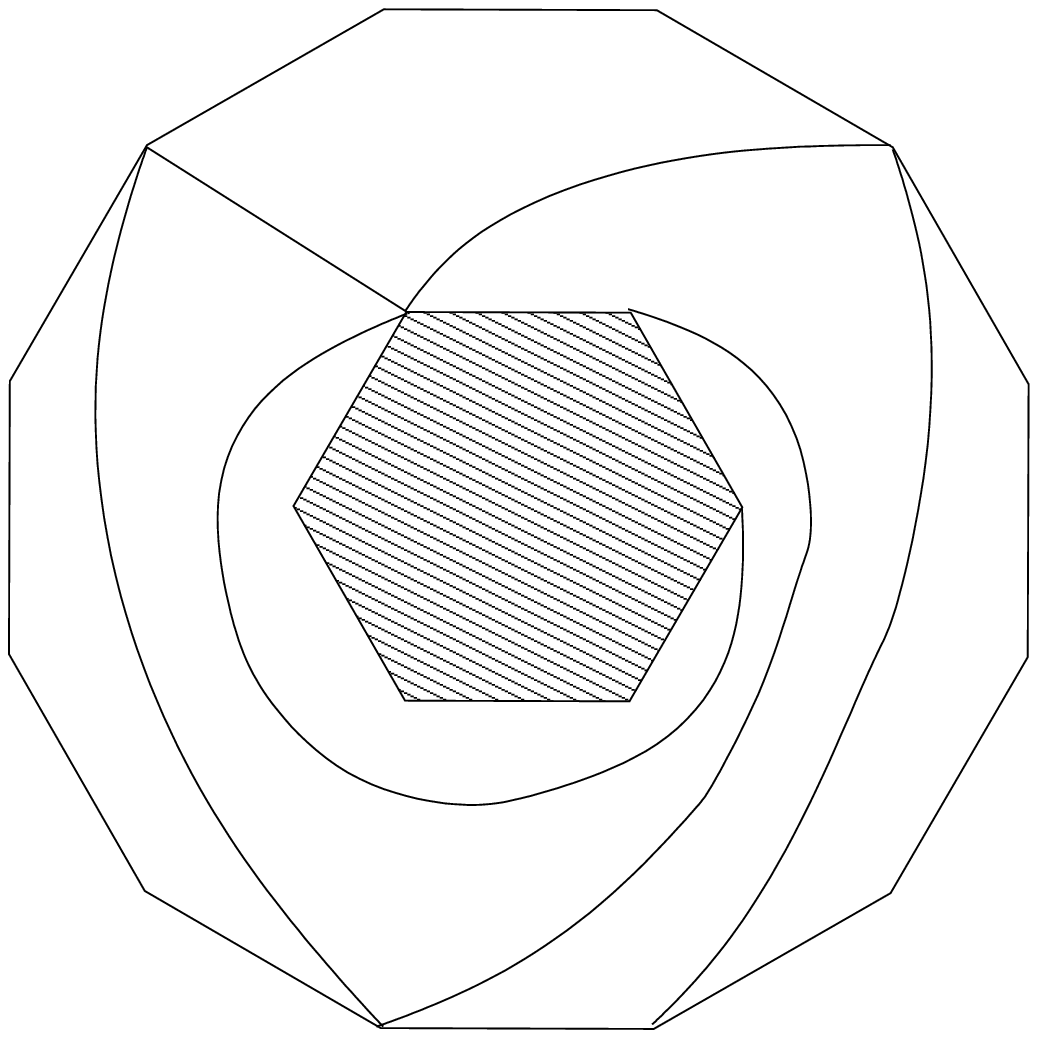}
  \end{center}\caption{\label{figexang} $4$-angulation of $P_{3,3,2}$
    and $5$-angulation of $P_{4,2,3}$.} 
  \end{figure}

Let $\alpha$ be a diagonal, and let $x$ be the point at the center of the
inner (and hence outer) polygon. The winding number of $\alpha$ is an
integer denoting how many times $\alpha$ travels around $x$. If
$\alpha$ travels around $s$ times but not $s+1$ times, we say that the
winding number is $s$. We do not care about orientation, so the
winding number is always $\geq 0$.

\section{The quiver corresponding to an $(m+2)$-angulation}\label{quiver}

For an $(m+2)$-angulation $\Delta$ of $P_{p,q,m}$, we define a
corresponding coloured quiver $Q_{\Delta}$ with $p+q$ vertices in the
following way. The vertices are the $m$-diagonals. There
is an arrow between $i$ and $j$ if the $m$-diagonals bound a common
$(m+2)$-gon. The colour of the arrow is the number of edges forming
the segment of the boundary of the $(m+2)$-gon which lies between $i$
and $j$, counterclockwise from $i$. This is the same definition as in
\cite{bt} in the Dynkin $A$ case, and it is easy to see that such a
quiver satisfy the conditions described in \cite{bt} for coloured
quivers. See Figure \ref{figannulusquiver} for an example. If $\alpha$
is an $m$-diagonal in an $(m+2)$-angulation, we always denote by
$v_{\alpha}$ the corresponding vertex in $Q_{\Delta}$. 

  \begin{figure}[htp]
  \begin{center}
    \includegraphics[width=4.6cm]{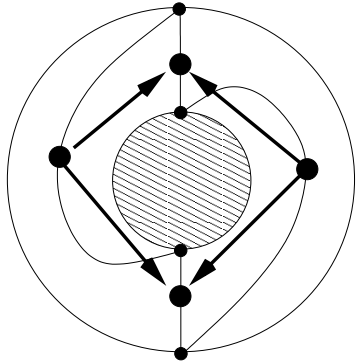}
  \end{center}\caption{\label{figannulusquiver} Example of a
    triangulation $\Delta$ of $P_{2,2}$ and its corresponding quiver
    $Q_{\Delta}$.} 
  \end{figure}

It is known from \cite{fst} that a quiver obtained in this way, for
$m=1$, is a quiver of a cluster-tilted algebra of type
$\widetilde{A}_{p,q}$. We also know that all quivers of cluster-tilted
algebras of type $\widetilde{A}_{p,q}$ can be obtained this way.

If $\alpha$ is a diagonal in a triangulation $\Delta$, the mutation of
$\Delta$ at $\alpha$ is the triangulation $\Delta'$ obtained by
replacing $\alpha$ with the unique other diagonal $\alpha'$ such that
$\Delta' = (\Delta - \alpha) \cup \alpha'$ is a triangulation. See
Figure \ref{figmutationexample}. It is known that this operation
commutes with quiver mutation, i.e. mutating at $\alpha$ corresponds
to mutating at $v_{\alpha}$.

  \begin{figure}[htp]
  \begin{center}
    \includegraphics[width=9cm]{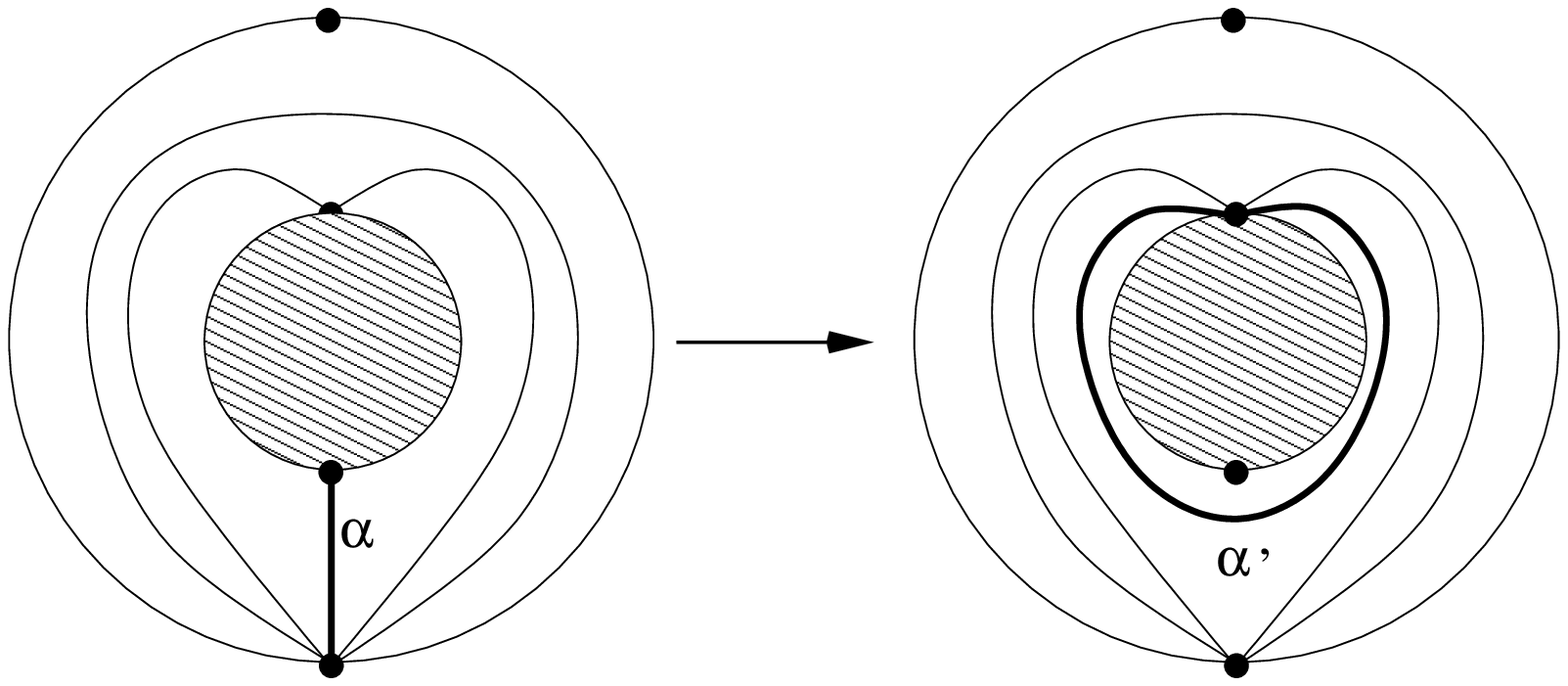}
$$\xymatrix{
&\bullet \ar[dr]&&&&\bullet \ar[dr]\ar[dl]&\\
v_{\alpha} \ar[ur]&&\bullet \ar[dl]&\longrightarrow&v_{\alpha'} \ar[dr]&&\bullet \ar[dl]\\
&\bullet \ar[ul]&&&&\bullet \ar[uu]&
}$$
  \end{center}\caption{\label{figmutationexample} Example of mutation
    of a triangulation at a diagonal and the corresponding quiver.} 
  \end{figure}

Let $\Delta$ be any $(m+2)$-angulation of $P_{p,q,m}$, and let $\alpha
\in \Delta$. By removing the $m$-diagonal $\alpha$ from $\Delta$, we
obtain an inner $(2m+2)$-gon in the "almost complete"
$(m+2)$-angulation $\Delta - \alpha$. There are exactly $m+1$ possible 
$m$-diagonals, say $\alpha=\alpha_0, \alpha_1, \alpha_2, ...,
\alpha_m$, such that $(\Delta - \alpha) \cup \alpha_i$ is an
$(m+2)$-angulation. These possible $m$-diagonals are called diameters
of the inner $(2m+2)$-gon, because they geometrically connect two
opposite vertices in the $(2m+2)$-gon. Consequently, they can all be
obtained from $\alpha$ by rotating the inner $(2m+2)$-gon. For an
$m+2$-angulation $\Delta$, we define the mutation at $\alpha$ to be
the $(m+2)$-angulation $\mu_{\alpha}(\Delta)$ obtained by rotating the
$(2m+2)$-gon corresponding to $\alpha$ clockwise. This definition is
similar to the same as the one given in \cite{bt} for the Dynkin case
$A$ and $(m+2)$-angulations of regular polygons. 

We note that this operation is well-defined, for if $\alpha$ is an
$m$-diagonal between $O_i$ and $I_j$, then $i$ and $j$ are congruent
modulo $m$, and of course $i-1$ is congruent to $j-1$ modulo $m$. It
is straightforward to check that the operation is well-defined on
$m$-diagonals of Type 2 and 3. Also, it is well-defined for mutation
that takes an $m$-diagonal of a certain type to an $m$-diagonal of
another type. 

Let us fix a particular $(m+2)$-angulation $\Delta_{p,q,m}^0$ of
$P_{p,q,m}$ and the corresponding quiver $Q_{\Delta_{p,q,m}^0}$. See
Figure \ref{figparticulartriquiv}. 
  \begin{figure}[htp]
  \begin{center}
    \includegraphics[width=10cm]{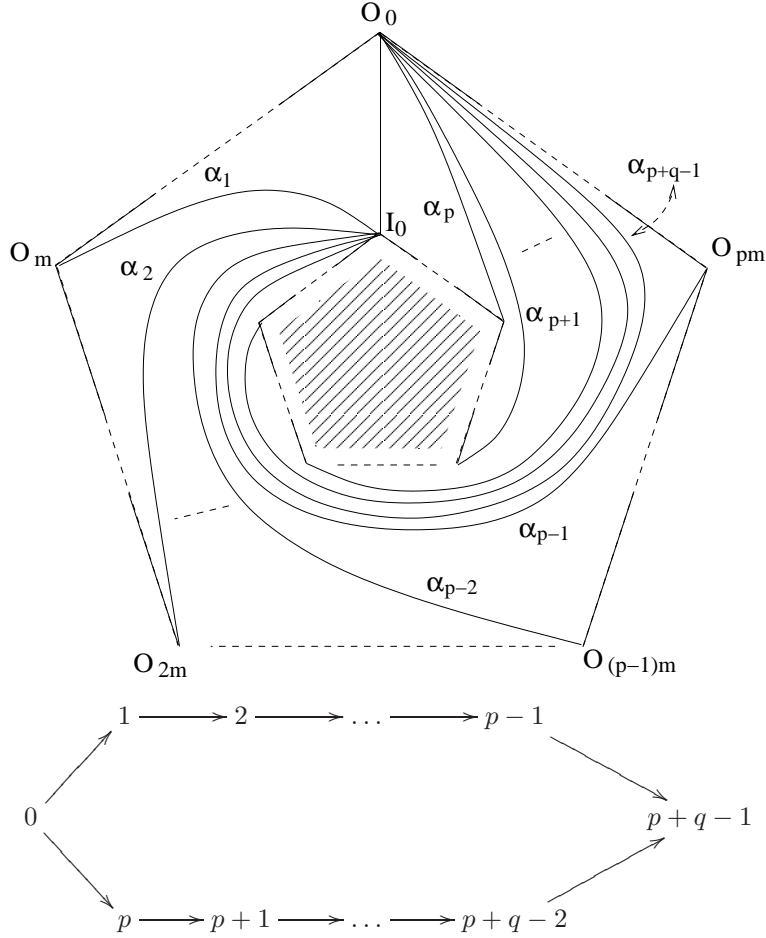}
$$\xymatrix{
&1\ar[r]&2\ar[r]&\hdots\ar[r]&p-1\ar[dr]\\
0\ar[ur]\ar[dr]&&&&&p+q-1\\
&p\ar[r]&p+1\ar[r]&\hdots\ar[r]&p+q-2\ar[ur]\\
}$$
  \end{center}\caption{\label{figparticulartriquiv} The triangulation
    $\Delta_{p,q,m}^0$ of $P_{p,q,m}$ and the corresponding quiver
    $Q_{\Delta_{p,q,m}^0}$. Note that we have only drawn the arrows of
    colour $0$, and there are only arrows of colour $0$ and $m$. Also
    note that all $m$-diagonals are of Type 1.}  
  \end{figure}
The $(m+2)$-angulation $\Delta^0_{p,q,m}$ gives rise to the coloured
quiver $Q_{\Delta^0_{p,q,m}}$ of type $\tilde{A}_{p,q}$. The proof of
the following proposition is straightforward and similar to the Dynkin
type $A$ case \cite{bt}, and we leave the proof to the reader.

\begin{prop}\label{mutationcommutes}
Mutation at any $m$-diagonal $\alpha$ in an $(m+2)$-angulation
$\Delta$, corresponds to coloured quiver mutation at $v_{\alpha}$ in
$Q_{\Delta}$. 
\end{prop}

Any $(m+2)$-angulation can be obtained from $\Delta^0_{p,q,m}$ by a 
finite sequence of mutations. This can be seen by noting that any 
$(m+2)$-angulation consisting of only $m$-diagonals of Type 1 can 
be reached from $\Delta^0_{p,q,m}$ by a finite sequence of mutations, 
and any $(m+2)$-angulation can be mutated into an $(m+2)$-angulation 
with $m$-diagonals of only Type 1. It follows that any coloured 
quiver of type $\tilde{A}_{p,q}$ can be obtained from an 
$(m+2)$-angulation and that a coloured quiver corresponding to an
$(m+2)$-angulation of $P_{p,q,m}$, is a coloured quiver of an
$m$-cluster-tilted algebra of type $\tilde{A}_{p,q}$.

Let $\mathcal{T}_{p,q,m}$ be the set of all $(m+2)$-angulations of
$P_{p,q,m}$, and let $\mathcal{M}_{p,q,m}$ be the mutation class of
$m$-coloured quivers of type $\widetilde{A}_{p,q}$. By the above we
have a surjective function $$\sigma_{p,q,m}: \mathcal{T}_{p,q,m}
\rightarrow \mathcal{M}_{p,q,m},$$ where $\sigma_{p,q,m}(\Delta) =
Q_{\Delta}$. The function $\sigma_{p,q,m}$ commutes with mutation.

\section{The category of $m$-diagonals}\label{categoryofdiagonals}

Let $\alpha$ be an $m$-diagonal in $P_{p,q,m}$ and $s$ a positive
integer. We define $\alpha[s]$ to be the $m$-diagonal obtained by
rotating the outer polygon $s$ steps clockwise and the inner polygon
$s$ step counterclockwise. More precisely,
\begin{itemize}
\item $\alpha[s]$, where $\alpha$ is a path from $O_i$ to $I_j$ of
  Type 1, is obtained by continuously moving the endpoint of the path
  at $O_i$ to $O_{i-s}$ and the endpoint of $I_j$ to $I_{j-s}$; 
\item $O_{i,k}[s] = O_{i-s,k}$;
\item $I_{i,k}[s] = I_{i-s,k}$. 
\end{itemize}

We always compute modulo $mp$ and $mq$ when we refer to vertices on
the outer and inner polygon respectively. Obviously we can define the
opposite operation, and we denote it by $[-s]$. Certainly this
operation is well-defined, for if $\alpha$ is an $m$-diagonal, then
$\alpha[s]$ is also an $m$-diagonal for all integers $s$. Set $\tau = 
[m]$. We have the following lemma, which follows directly from the
definition.

\begin{lem}\label{lemmatau}
We have that $O_{i,k}[mp]=\tau^p O_{i,k} = O_{i,k}$ and
$I_{i,k}[mq]=\tau^q I_{i,k} = I_{i,k}$. Furthermore, $\alpha[s] \neq
\alpha$ for all $s$, when $\alpha$ is of Type 1.   
\end{lem}

If $\Delta$ is an $(m+2)$-angulation, it is clear that if $\Delta[s]$
is the $(m+2)$-angulation obtained from $\Delta$ by applying $[s]$ on
each diagonal in $\Delta$, we obtain a new $(m+2)$-angulation. It is
also clear that $Q_{\Delta} = Q_{\Delta[s]}$ for all $s$, since
$m$-diagonals bounding a common $(m+2)$-gon in $\Delta$ also bound a
common $(m+2)$-gon in $\Delta[s]$. The function $\sigma_{p,q,m}:
\mathcal{T}_{p,q,m} \rightarrow \mathcal{M}_{p,q,m}$ from the previous
section is therefore not an injection.

We want to define a category of $m$-diagonals, and the construction is
motivated by \cite{ccs}, where they defined the cluster category of
type $A_n$ using diagonals of regular polygons. This construction was
generalized in \cite{bm2} to $m$-cluster categories.

First we define elementary moves of $m$-diagonals, which are certain
operations that send one $m$-diagonal to another $m$-diagonal, and it
is easy to check that the operation is well-defined. The operation
should be considered as continuously moving the endpoints of the
$m$-diagonals. We consider several cases.   

\begin{itemize}
\item $m$-diagonals of Type 1: If $\alpha$ is an $m$-diagonal between $O_i$ and
  $I_j$, there are exactly two elementary moves:
  $$\xymatrix@R=0.3pc{
    &\beta\\
    \alpha\ar[ur]\ar[dr]&\\
    &\epsilon
  }$$
  The $m$-diagonal $\beta$ is the $m$-diagonal obtained from $\alpha$
  by continuously moving the endpoint of $\alpha$ at $O_i$
  counterclockwise $m$ steps to $O_{i+m}$. The $m$-diagonal $\epsilon$
  is the $m$-diagonal obtained from $\alpha$ by continuously moving
  the endpoint of $\alpha$ at $I_j$ clockwise $m$ steps to $I_{j+m}$. 
\item $m$-diagonals of Type 2: 
  \begin{itemize}
    \item If $k=m+2$, there is exactly one elementary move, 
      $$O_{i,k} \rightarrow O_{i,k+m}.$$
    \item If $k > m+2$, there are exactly two elementary moves:
    $$\xymatrix@R=0.3pc{
       &O_{i,k+m}\\
       O_{i,k}\ar[ur]\ar[dr]&\\
       &O_{i+m,k-m}
    }$$
  \end{itemize}
\item $m$-diagonals of Type 3:
  \begin{itemize}
    \item If $k=m+2$, there is exactly one elementary move, 
      $$I_{i,k} \rightarrow I_{i,k+m}.$$
    \item If $k > m+2$, there are exactly two elementary moves:
    $$\xymatrix@R=0.3pc{
       &I_{i,k+m}\\
       I_{i,k}\ar[ur]\ar[dr]&\\
       &I_{i+m,k-m}
    }$$
  \end{itemize}
\end{itemize}

See Figure \ref{figelementarymoves} for several examples of elementary
moves. 

  \begin{figure}[htp]
  \begin{center}
    \includegraphics[width=5cm]{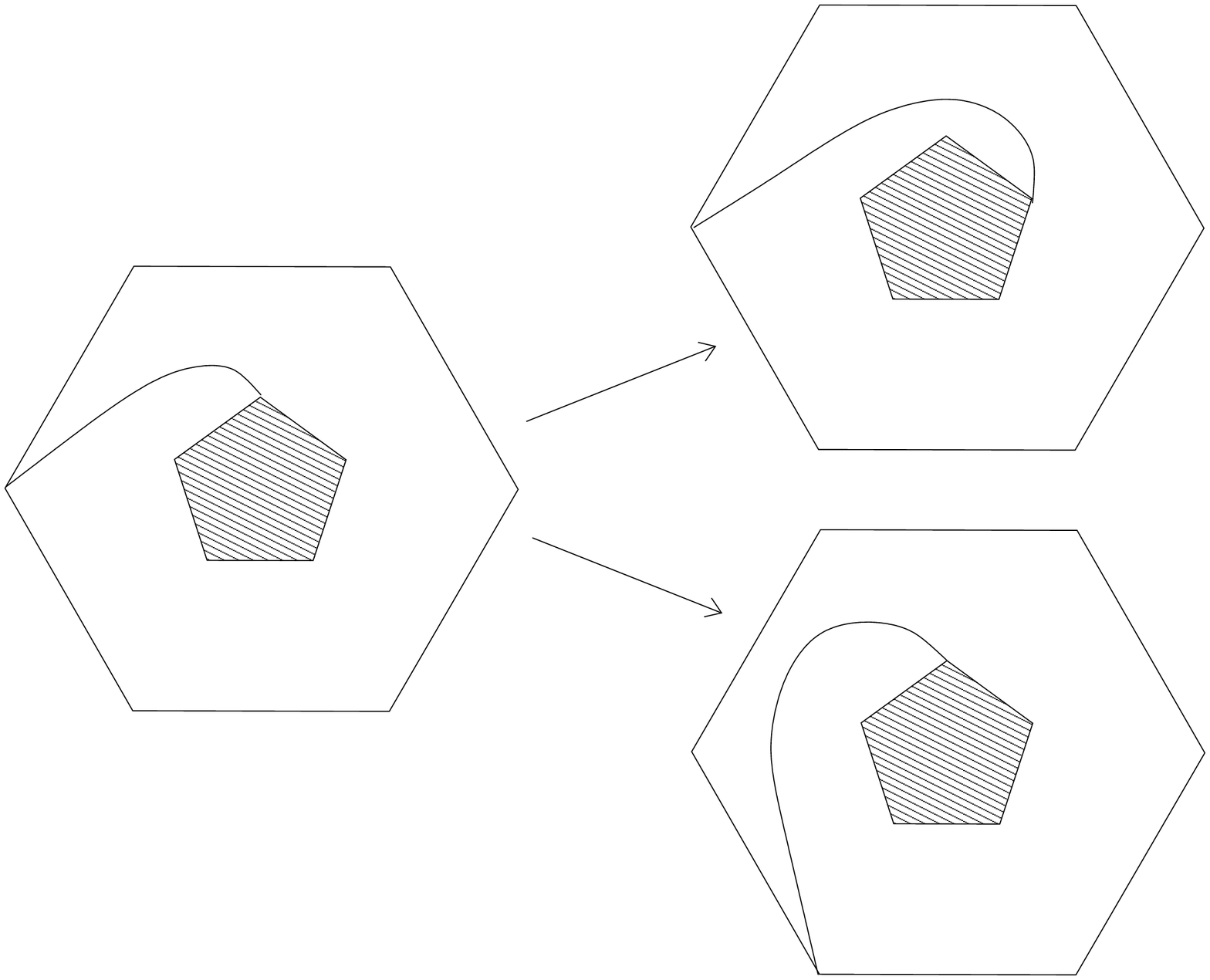}
    \includegraphics[width=5cm]{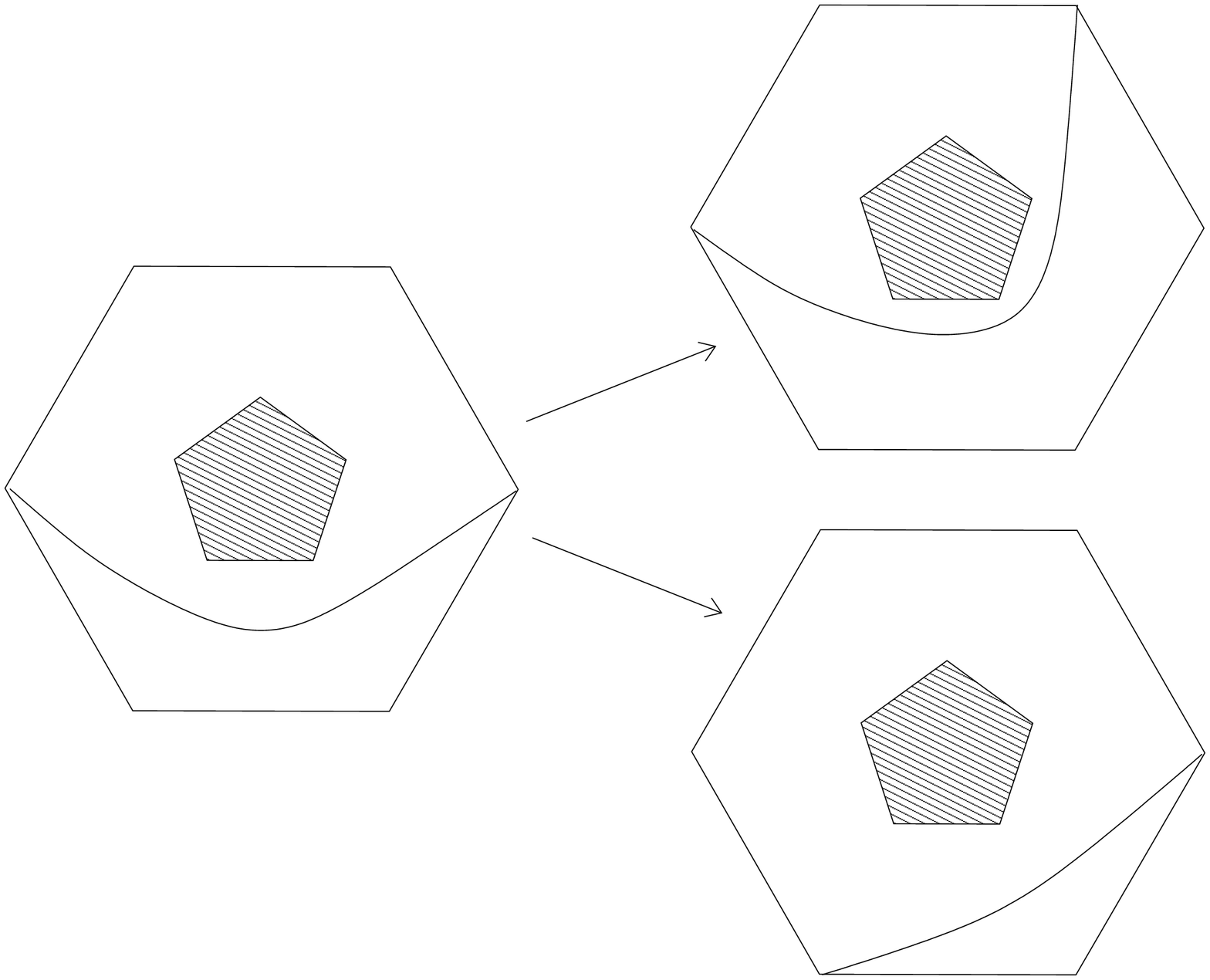}
    \includegraphics[width=5cm]{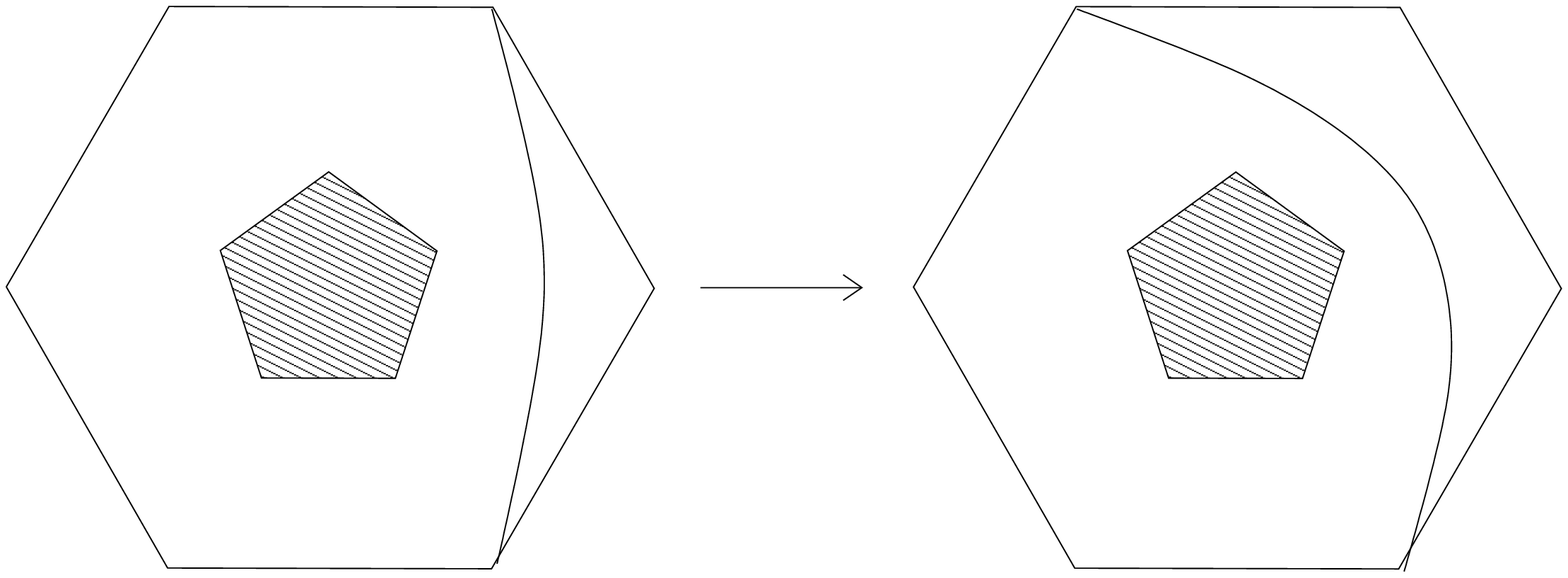}
    \includegraphics[width=5cm]{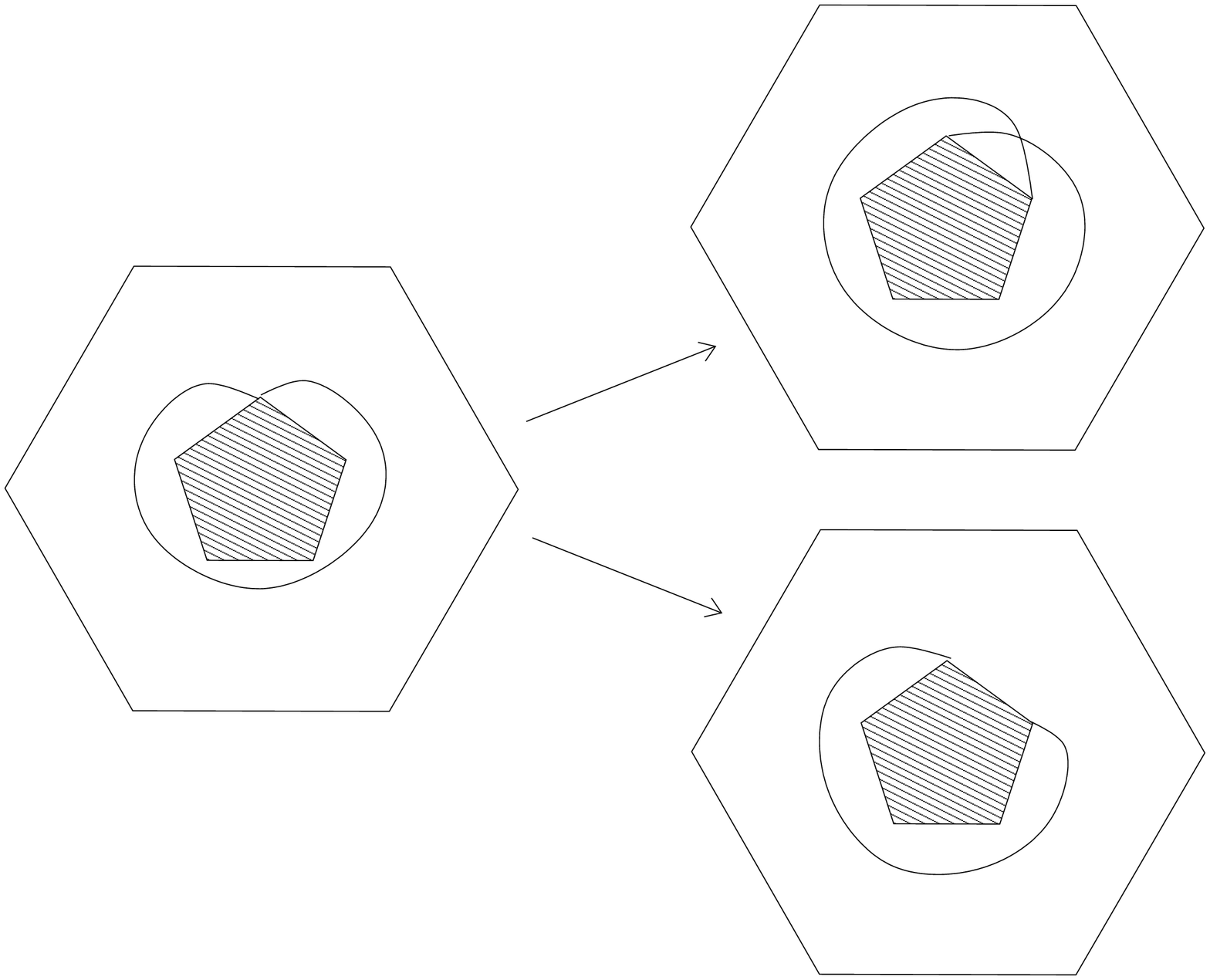}
  \end{center}\caption{\label{figelementarymoves} Examples of
    elementary moves for $m=1$.} 
  \end{figure}

We can think of elementary moves of $m$-diagonals of Type 1 as rotating
the outer and inner polygon $m$ steps counterclockwise and clockwise
respectively. Note that rotating only the outer polygon gives rise to
another $(m+2)$-angulation that preserves the corresponding coloured
quiver. Similarly for the inner polygon. %We will investigate this in
%Section \ref{equivalences}.   

We need the following proposition.

\begin{prop}
Let $\alpha$ and $\beta$ be $m$-diagonals. Then there is an elementary 
move $\alpha \rightarrow \beta$ if and only if there is an elementary
move $\tau \beta \rightarrow \alpha$.
\end{prop}
\begin{proof}
Suppose $f:\alpha \rightarrow \beta$ is an elementary move. 

First, assume that $\alpha$ and $\beta$ are of Type 1, say $\alpha$
is an $m$-diagonal between $O_i$ and $I_j$ and $\beta$ is an
$m$-diagonal between $O_{i'}$ and $I_{j'}$. Then either $i'=i+m$ and
$j'=j$ or $i'=i$ and $j'=j+m$. If $i'=i+m$ and $j'=j$, then $\tau
\beta$ is an $m$-diagonal between $O_i$ and $I_{j-m}$. Then, by
definition, there is an elementary move $\tau \beta \rightarrow
\alpha$. Similarly when $i'=i$ and $j'=j+m$. 

Next, suppose $\alpha$ and $\beta$ are of Type 2, say $\alpha =
O_{i,k}$ and $\beta = O_{i',k'}$. If $k=m+2$, we have that $i'=i$ and
$k' = k+m = 2m+2$. Then $\tau \beta = \tau O_{i',k'} = \tau O_{i,
  2m+2} = O_{i-m,2m+2}$, and by definition, there is an elementary
move $\tau \beta = O_{i-m,2m+2} \rightarrow
\alpha=O_{i,2m+2-m}=O_{i,m+2} =O_{i,k}$. 

If $k \geq m+2$, then either $i'=i$ and $k'=k+m$ or $i'=i+m$ and
$k'=k-m$. If $i'=i$ and $k'=k+m$, we have $\tau \beta = \tau O_{i',k'}
= \tau O_{i,k+m} = O_{i-m,k+m}$. Then, by definition, there is an
elementary move $\tau \beta = O_{i-m,k+m} \rightarrow \alpha =
O_{i,k}$. Similarly if $i'=i+1$ and $k'=k-1$.  

In the same way we can show that this holds for $m$-diagonals of Type
3. The converse is similar. 
\end{proof}

Let $K$ be an algebraically closed field, and let $C^m_{p,q}$
be the $K$-linear additive category defined as follows. The
indecomposable objects are the $m$-diagonals, so the objects in
$C^m_{p,q}$ are direct sums of the $m$-diagonals. Morphisms 
between indecomposable objects $X$ and $Y$ are the vector space over
$K$ spanned by the elementary moves modulo certain mesh relations
which we define below. 

Let $\alpha$ be an indecomposable object (an $m$-diagonal) in
$C^m_{p,q}$. If $f: \beta \rightarrow \alpha$ is an elementary
move, there exist an elementary move $f': \tau \alpha \rightarrow
\beta$ by the proposition. Let $L$ be the set of all elementary moves
ending in $\alpha$. Then the mesh relation is defined as $$\sum_{f \in
  L} f f'.$$

Consider the following situation, where $f_1,f_2,...,f_n$ are all
elementary moves ending in $\alpha$.

\[\xymatrix@R=0.9pc@C=7pc{
&\beta_1\ar[ddr]^{f_1}&\\
&\beta_2\ar[dr]^{f_2}&\\
\tau\alpha\ar[uur]^{f_1'}\ar[ur]^{f_2'}\ar[dr]^{f_t'}&\vdots&\alpha\\
&\beta_t\ar[ur]^{f_t}&\\
}\]

This means that the sum of compositions $f_i f_i'$ is $0$. In our case
there are at most two elementary moves, so $t=1$ or $t=2$. If $t=1$,
the diagonals are of Type 2 ($O_{i,k}$) or Type 3 ($I_{i,k}$) and
$k=m+2$. This means that the compositions $O_{i,m+2} \rightarrow
O_{i,2m+2} \rightarrow O_{i+m,m+2}$ and $I_{i,m+2} \rightarrow
I_{i,2m+2} \rightarrow I_{i+m,m+2}$ are $0$. See Figure
\ref{figcompzero}. If $t=2$, we have equalities of compositions of
elementary moves. See Figure \ref{figcompzero2}.

  \begin{figure}[htp]
  \begin{center}
    \includegraphics[width=11.5cm]{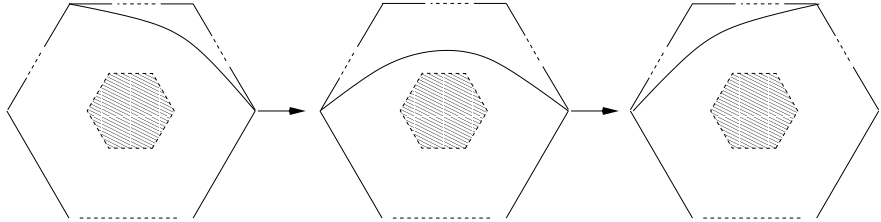}
  \end{center}\caption{\label{figcompzero} The composition $O_{i,m+2}
    \rightarrow O_{i,2m+2} \rightarrow O_{i+m,m+2}$ is $0$.}   
  \end{figure}

  \begin{figure}[htp]
  \begin{center}
    \includegraphics[width=11.5cm]{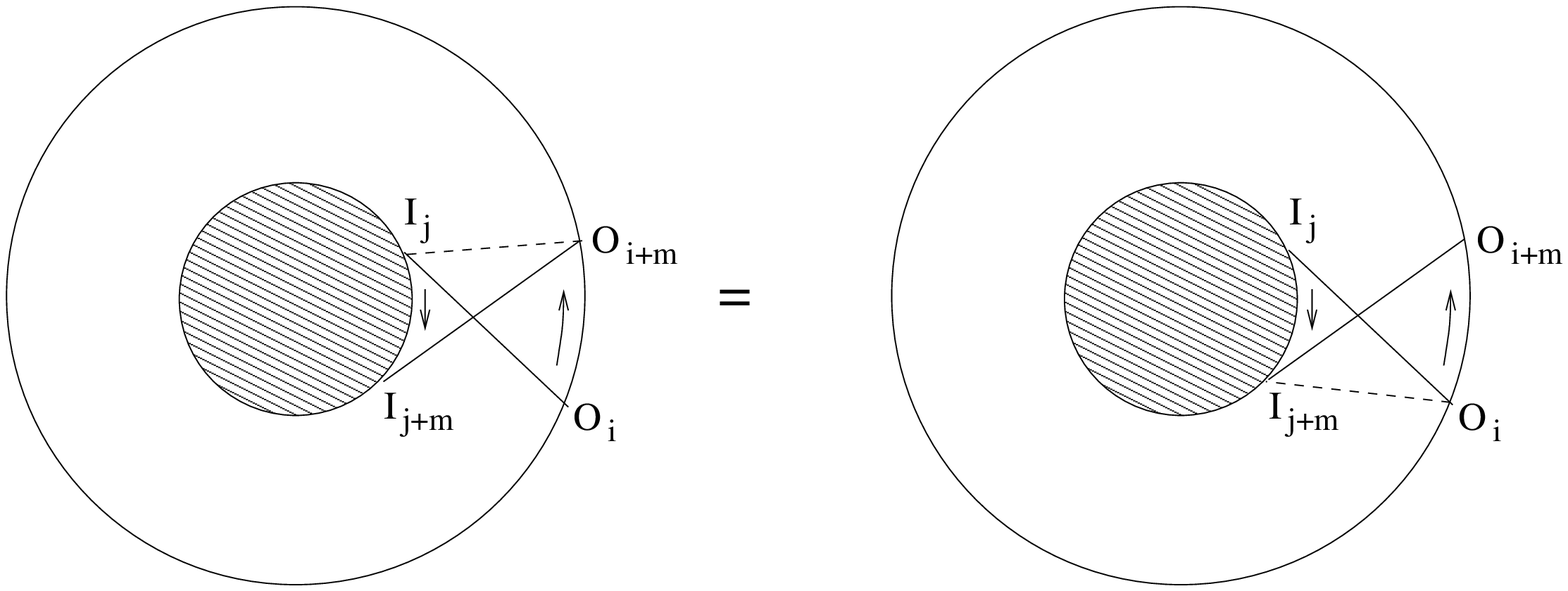}
    \includegraphics[width=11.5cm]{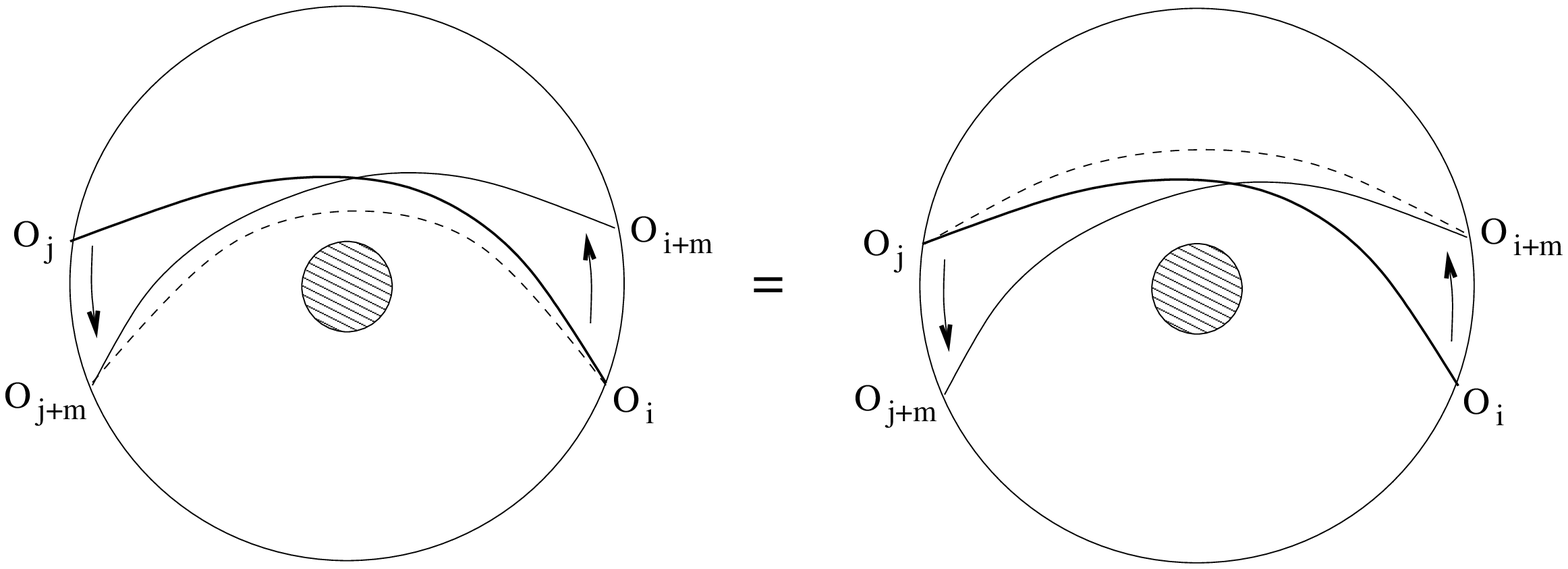}
  \end{center}\caption{\label{figcompzero2} Equalities of compositions
  of elementary moves of diagonals of Type 1 and Type 2
  respectively. The dotted line is the intermediate $m$-diagonal in
  the composition.}   
  \end{figure}

The translation $\tau$ is clearly an equivalence on this category. In
fact $[s]$ is an equivalence for all $s$.

\section{The $m$-cluster category and the category of
  $m$-diagonals}\label{arquiver}
Given $P_{p,q,m}$ and the category of $m$-diagonals $C^m_{p,q}$, we
define a quiver in the following way. The vertices are the
indecomposable objects (i.e. the $m$-diagonals), and there is an arrow
from the indecomposable object $\alpha$ to the indecomposable object
$\beta$ if an only if there is an elementary move $\alpha \rightarrow
\beta$. We call this the AR-quiver of $C^m_{p,q}$, and we will see
that it is isomorphic to a subquiver of the AR-quiver of the
$m$-cluster category obtained from $\widetilde{A}_{p,q}$.

Let $\alpha$ be an $m$-diagonal. If $\alpha$ is of Type 1, we say that
$\alpha$ is in level $d$, where $0 \leq d <m$, if we can obtain the
$m$-diagonal between $O_d$ and $I_d$ with winding number $0$ from
$\alpha$ by applying a finite sequence of $\tau$ and elementary
moves. If $\alpha$ is of Type 2 (or 3), we say that $\alpha$ is in
level $d$ if we can obtain the $m$-diagonal $O_{d,m}$ (or $I_{d,m}$)
from $\alpha$ by a finite sequence of $\tau$ and elementary moves.  

It is straightforward to show that every $m$-diagonal is in some
level. Given an $m$-diagonal in level $d$, we can not reach another
$m$-diagonal in a level $\neq d$ by a finite sequence of $\tau$ and
elementary moves. It follows that the quiver of $m$-diagonals consists
of at least $m$ components. It is also clear that there is no sequence
of $\tau$ and elementary moves between $m$-diagonals of different
types. Also, there exist a sequence of elementary moves and $\tau$
between any two $m$-diagonals in the same level and of the same
type. In fact, if $\alpha$ is of Type 1 in level $d$, then $\alpha$ is
of the form $\tau^s \alpha'[d]$, where $\alpha' \in
\Delta^0_{p,q,m}$. It follows that the quiver consists of $3m$
components. It is easy to see that the components consisting of
$m$-diagonals of the same type are isomorphic.  

Let $T^d_p$ be the component containing objects of Type 2 in level
$d$ and let $T^d_q$ be the component containing objects of type Type 3
in level $d$. Denote by $S^d$ the component consisting of
$m$-diagonals of Type 1 in level $d$. See Figure \ref{figarquiver}
for an example. We draw the translation $\tau$ as dotted arrows. 

Now we want to define an additive functor $F : C^m_{p,q} \rightarrow
\mathcal{C}^m$, where $\mathcal{C}^m$ is the $m$-cluster category of
type $\widetilde{A}_{p,q}$. It is enough to define the functor on
indecomposable objects and elementary moves. We also want that $F$
induces a quiver isomorphism between the AR-quiver of the category of
$m$-diagonals and a subquiver of the AR-quiver of the $m$-cluster
category. 
%Furthermore, we want that $F$ commutes with mutation,
%i.e. $F(\mu_i \Delta) = \mu_i F(\Delta)$ for all $(m+2)$-angulations
%$\Delta$.  

First we consider the objects in the components $S^d$. Denote by
$\mathcal{S}^d$ the component in the AR-quiver of the $m$-cluster
category consisting of objects of the form $\tau_{\mathcal{C}}^s
P[d]$, where $P$ is a projective and $\tau_{\mathcal{C}}$ is the
Auslander-Reiten translation. We want to show that $S^d$ is
isomorphic to $\mathcal{S}^d$ via the functor $F$, which we will
define below. When there is no confusion, we write the AR-translation
$\tau_{\mathcal{C}}$ as $\tau$. Similarly with the shift functor $[i]$.

First we make a choice, and it is natural to let the
$(m+2)$-angulation $\Delta^0_{p,q,m}$ and the quiver
$Q_{\Delta^0_{p,q,m}}$ in Figure \ref{figparticulartriquiv}
correspond to the $m$-cluster tilting object $T=\oplus P_i$, where
$P_i$ is the projective corresponding to vertex $i$ in the quiver. The
$m$-diagonal $\alpha_i$ in $\Delta^0_{p,q,m}$ corresponding to vertex
$i$ is mapped to $P_i$. Also, we define $F(\tau^s \alpha_i) =
\tau^s P_i$ for all integers $s$. By Lemma
\ref{lemmatau} and by the fact that all elements in $S^0$ are of the
form $\tau^s \alpha_i$, this gives a bijection between the
$m$-diagonals in $S^0$ and the objects in the component
$\mathcal{S}^0$ of the AR-quiver of the $m$-cluster category.  

%Now, if we mutate $T$ in direction $P_i$, we obtain the $m$-cluster
%tilting object $\mu_i (T) = (T / P_i) \oplus P_i[1]$. We would like that
%this corresponds to mutating at the $m$-diagonal $\alpha_0 \in
%\Delta^0_{p,q,m}$. Let $\beta$ be the $m$-diagonal between $O_1$ and
%$I_1$, which we get by mutating at $\alpha_0$. It follows that, to
%have any hope that $F$ commutes with mutation, $F(\beta) =
%P_0[1]$. 
%Note that $\beta = \alpha_0[1]$. Continuing in this way, we
%set 
Next we define %$F(\alpha_0[d]) = P_{0}[d]$, and in general
$F(\alpha_i[d]) = P_{i}[d]$ and $F(\tau^t \alpha_i[d]) =
\tau^t P_i[d]$ for all integers $t$. This takes care of
all $m$-diagonals of Type 1, and by Lemma \ref{lemmatau} and by the
fact that all elements in $S^d$ are of the form $\tau^t \alpha_i[d]$,
this is a bijection between the set of $m$-diagonals of Type 1 and the
set of indecomposable objects in the transjective components in the
AR-quiver of the $m$-cluster category. 

Next we want to show that elementary moves (or arrows in the quiver of
$m$-diagonals) correspond to irreducible morphisms (arrows in the
AR-quiver of the $m$-cluster category). Suppose there is an
elementary move between two $m$-diagonals $\alpha \rightarrow
\beta$. Then there exist an integer $s$ such that $\tau^s \alpha$ is
of the form $\alpha_i[d]$ for some integer $d$ and $\alpha_i \in
\Delta^0_{p,q,m}$. Also, there is an elementary move $\tau^s \alpha 
\rightarrow \tau^s \beta$. We have $F(\tau^s \alpha) =
F(\alpha_i[d])=P_i[d]$, where $P_i$ is the projective corresponding to
the vertex $i$ in $Q_{\Delta^0_{p,q,m}}$. By considering the
possibilities for $\tau^s \beta$, it is now straightforward to check
that there is an arrow $F(\tau^s \alpha) = P_i[d] \rightarrow F(\tau^s
\beta)$, and hence there is an arrow $F(\alpha) \rightarrow
F(\beta)$. We leave the details to the reader. The converse is
similar. 

\begin{prop}
The component $S^d$ in the AR-quiver of $\mathcal{C}^m_{p,q}$ is
isomorphic to the component $\mathcal{S}^d$ in the AR-quiver of the
$m$-cluster category $\mathcal{C}^m$ of type $\widetilde{A}_{p,q}$,
for all integers $d$, with $0 \leq d < m$
\end{prop}

  \begin{figure}[htp]
  \begin{center}
    \includegraphics[width=12.5cm]{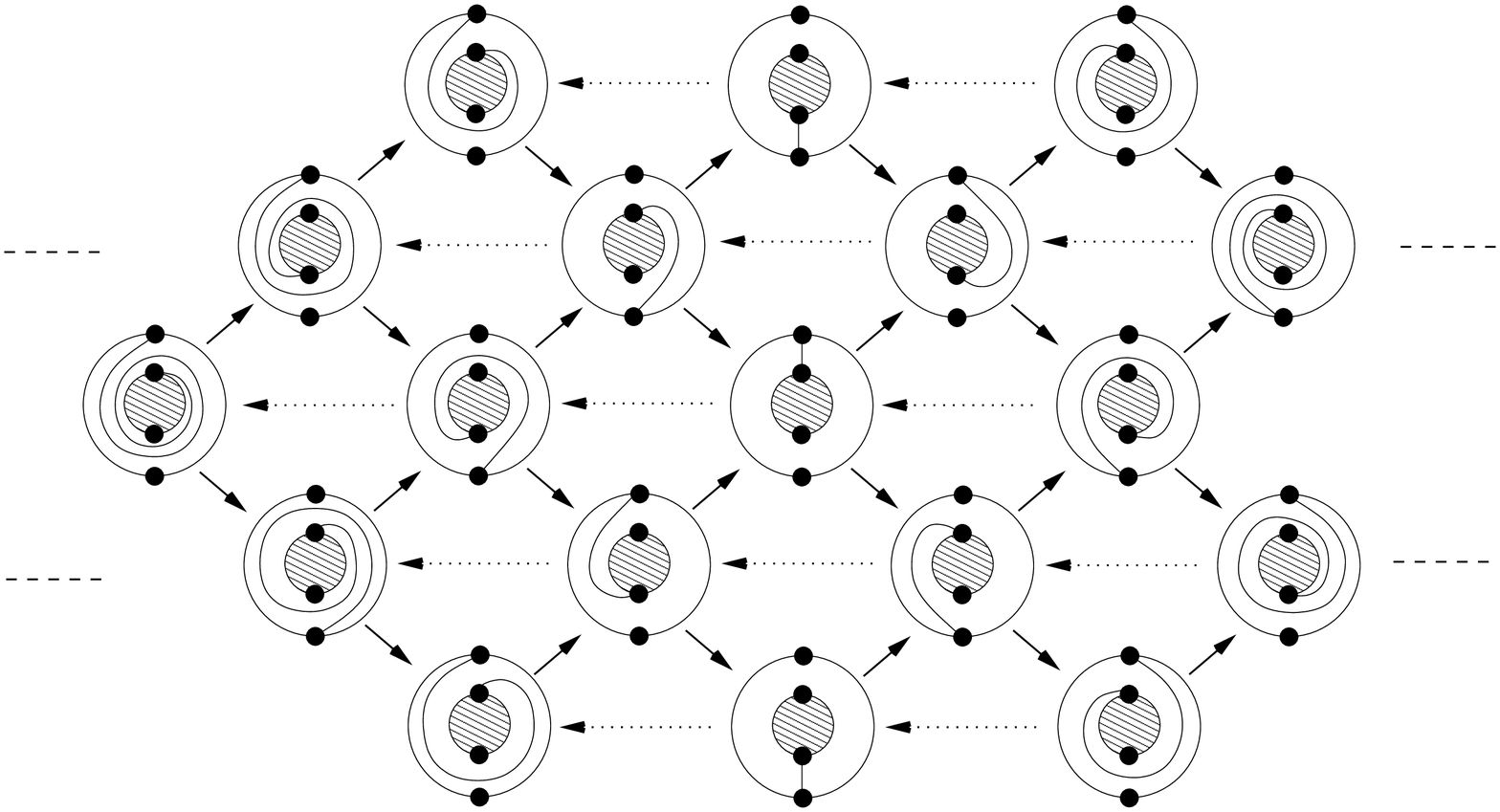}
    \includegraphics[width=9cm]{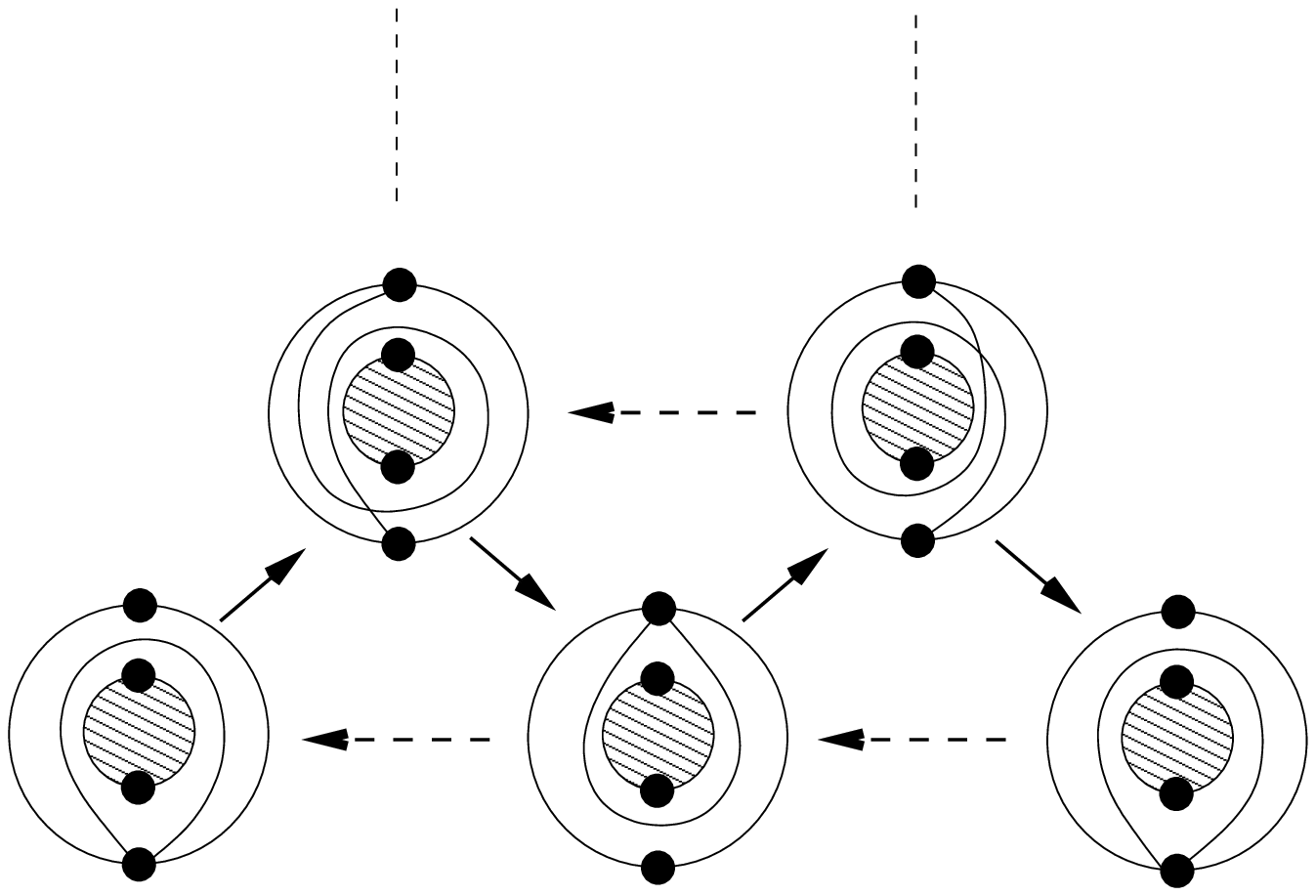}
    \includegraphics[width=9cm]{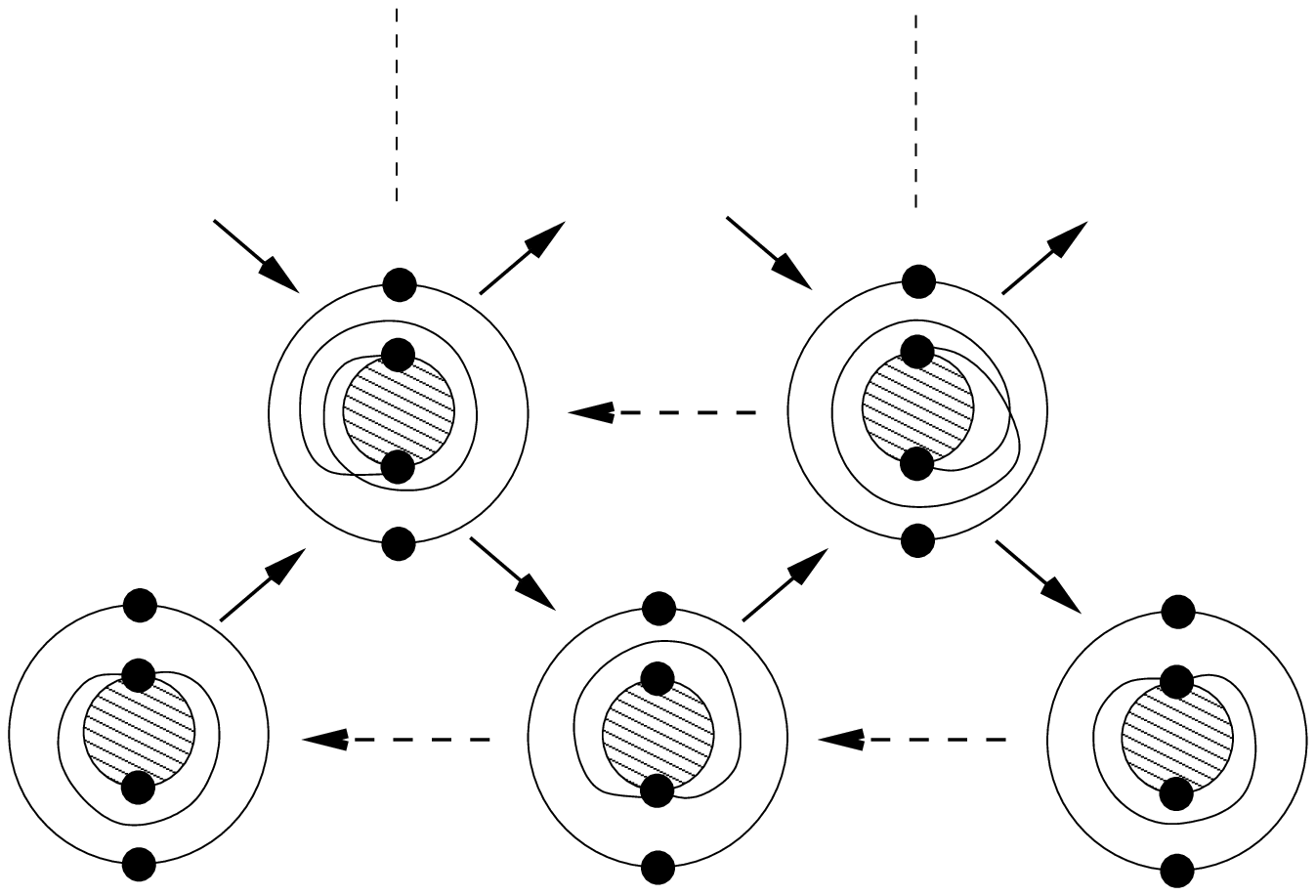}
  \end{center}\caption{\label{figarquiver} The quiver of
    $\mathcal{C}^1_{2,2}$; component $S$, $T_{p=2}$ and $T_{q=2}$
    respectively.}   
  \end{figure}

Now we consider the objects in the components $T_p^d$ and $T_q^d$ in
the AR-quiver of $C^m_{p,q}$. Let $\mathcal{T}^d_p$ and
$\mathcal{T}^d_q$ be the tubes of rank $p$ and $q$ respectively in the
$m$-cluster category of type $\tilde{A}_{p,q}$, where $d$ denotes the
degree. The tubes of rank $p$ are isomorphic, and they are isomorphic
to $\mathbb{Z}A_{\infty}/(\tau^p)$. The tubes of rank $q$ are isomorphic
to $\mathbb{Z}A_{\infty}/(\tau^q)$. We know 
that the tube $\mathcal{T}^1_p$ has $p$ quasi-simple objects, say
$Q^1_0, Q^1_1,..., Q^1_{p-1}$. Then the quasi-simple objects in
$\mathcal{T}^d_p$ are of the form $Q^1_0[d], Q^1_1[d],...,
Q^1_{p-1}[d]$. There is only one arrow to and one arrow from each of these
objects in the AR-quiver. Let the indecomposable objects in
$\mathcal{T}^d_p$ be denoted $Q_i^s[d]$, where $s$ is the quasi-length of 
$Q_i^s$. For a given $i$, a ray is the sequence of irreducible
maps $$Q^1_i[d] \rightarrow Q^2_i[d] \rightarrow Q^3_i[d] \rightarrow
Q^4_i[d] \rightarrow \hdots,$$ and a coray is the sequence $$\hdots
\rightarrow Q^4_{i-3}[d] \rightarrow Q^3_{i-2}[d]  \rightarrow
Q^2_{i-1}[d] \rightarrow Q^1_i[d].$$ We are in the situation shown in
Figure \ref{figarquivertube}. 
%We do similarly for the tubes $\mathcal{T}^d_q$. 

\begin{figure}[htp]
  \begin{center}
\[\xymatrix@C=0.5pc{
&\vdots&&&&\vdots&&&&\vdots&\\
&Q^4_{p-1}\ar[dr]&&Q^4_0\ar@{.>}[ll]\ar[dr]&&...&&Q^4_{p-3}\ar[dr]&&Q^4_{p-2}\ar@{.>}[ll]\ar[dr]&\\
Q^3_{p-1}\ar[dr]\ar[ur]&&Q^3_0\ar@{.>}[ll]\ar[dr]\ar[ur]&&Q^3_1\ar@{.>}[ll]&...&Q^3_{p-3}\ar[dr]\ar[ur]&&Q^2_{p-2}\ar@{.>}[ll]\ar[dr]\ar[ur]&&Q^2_{p-1}\ar@{.>}[ll]\\
&Q^2_0\ar[dr]\ar[ur]&&Q^2_1\ar@{.>}[ll]\ar[dr]\ar[ur]&&...&&Q^2_{p-2}\ar[dr]\ar[ur]&&Q^2_{p-1}\ar@{.>}[ll]\ar[dr]\ar[ur]&\\
Q^1_0\ar[ur]&&Q^1_1\ar@{.>}[ll]\ar[ur]&&Q^1_2\ar@{.>}[ll]&...&Q^1_{p-2}\ar[ur]&&Q^1_{p-1}\ar@{.>}[ll]\ar[ur]&&Q^1_0\ar@{.>}[ll]
}\]
  \end{center}\caption{\label{figarquivertube} A tube $\mathcal{T}_p$ of rank $p$ in
  the cluster category.}  
  \end{figure}
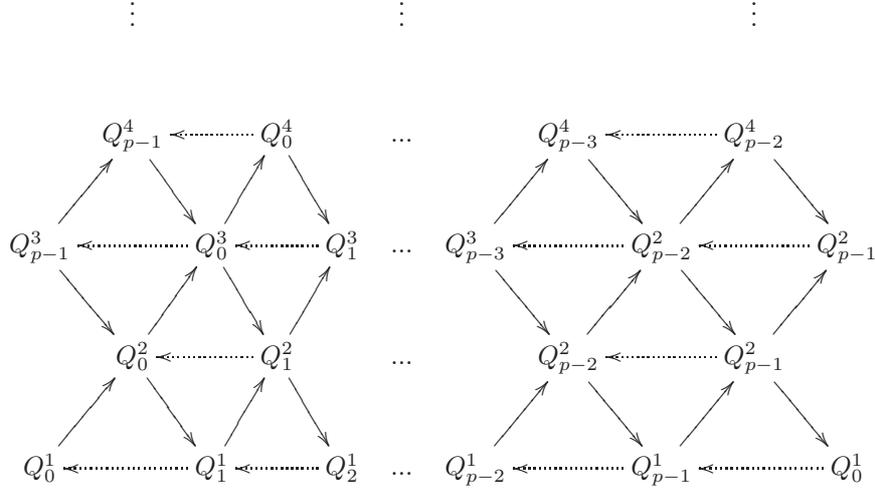

For a fixed $d$, with $0 \leq d <m$, we have that $\tau Q^s_i[d] =
Q^s_{i-1}[d]$, and that $\tau^p Q^s_i[d] = Q^s_i[d]$ for all $s$ and
$i$. It is clear, by Lemma \ref{lemmatau}, that the quasi-simple
objects have to correspond to diagonals of the form $O_{i,m+2}$, since
there are exactly one elementary move from and to $O_{i,m+2}$ and
their $\tau$-orbit is exactly $p$. 

For a fixed $d$ , we define $F(O_{0,m+2}[d]) = Q^1_0[d]$, $F(O_{0,2m+2}[d]) =
Q^2_0[d]$ and so on, i.e. $F(O_{0,jm+2}[d]) = Q^j_0[d]$. This corresponds to
the ray starting in $Q^1_0[d]$. We have $\tau^{-1} O_{0,m+2}[d] =
O_{m,m+2}[d]$, so we set $F(\tau^{-1} O_{0,m+2}[d]) = F(O_{m,m+2}[d])
= \tau^{-1} Q^1_0[d] = Q^1_1[d]$. Then we continue as above, and %set
%$F(\tau^{-1} O_{0,jm+2}[d]) = F(O_{m,jm+2}[d]) = \tau^{-1} Q^j_0[d] =
%Q^j_1$. The %
the pattern is clear, so in general we define $F(O_{im,jm+2}[d])
= Q_i^{j}[d]$. This takes care of all $m$-diagonals of Type 2. We do
similarly with the $m$-diagonals of Type 3, and they correspond to
objects in the tubes of rank $q$. We leave the details to the reader. 

Next we show that we have an irreducible morphism $Q^s_i[d] \rightarrow
Q^t_j[d]$ if and only if there is an elementary move $O_{im, sm+2}[d]
\rightarrow O_{jm,tm+2}$.

If $Q^s_i[d]$ is quasi-simple, $s=1$, and so
hence $t=2$ and $j=i$. Then we have an elementary move 
$$O_{im,sm+2}[d] = O_{im,m+2}[d] \rightarrow O_{jm,tm+2}[d]=O_{im,2m+2}[d].$$ 

If $Q^s_i$ is not quasi-simple, we have $s>1$, $t=s+1$ and $j=i$ or
$t=s-1$ and $j=i+1$. Then we have elementary moves 
$$O_{im,sm+2}[d] \rightarrow O_{jm,tm+2}[d]=O_{im,(s+1)m+2}[d] = O_{im,(sm+2+m)}[d]$$ 
and 
$$O_{im,sm+2}[d] \rightarrow O_{jm,tm+2}[d]=O_{(i+1)m,(s-1)m+2}[d] =
O_{im+m, sm+2-m}[d].$$   

The converse is similar.

We see that given any object $Q_i^s$, which corresponds to
$O_{im,sm+2}$, we have that $\tau Q_i^s = Q_{i-1}^s$ corresponds to
$\tau O_{im,sm+2} = O_{im-m,sm+2}$.

%Similarly we get that the the diagonals $I_{i,3}$ corresponds to the
%quasi-simple objects in the tube of rank $q$. In the same way we can
%prove that $T_q$ is isomorphic to $\mathcal{T}_q$ by letting objects
%$I_{i,k}$ correspond to objects $R_{i}^{k-2}$ in the tube of rank $q$.

\begin{prop}
The components $T^d_p$ and $T^e_q$ of the quiver of $\mathcal{C}_{p,q}$
are isomorphic to $\mathcal{T}^d_p$ and $\mathcal{T}^e_q$ respectively
in the $m$-cluster category.  
\end{prop}

It is now easy to see that $F$ is in fact a functor. We know that the
AR-quiver of the $m$-cluster category has tubes of rank 1,
i.e. non-rigid homogeneous objects. The functor $F$ is therefore not
dense. However, if $\mathcal{C}'$ is the full subcategory of the
$m$-cluster category $\mathcal{C}^m$ consisting of all objects not in a
homogeneous tube, we get the induced functor $G:\mathcal{C}^m_{p,q}
\rightarrow \mathcal{C'}$. Then $G$ is dense and faithful. The functor
is not full, since we do not have any maps in $C^m_{p,q}$ between the
objects in different components in the AR-quiver, i.e. no maps
corresponding to maps in the infinite radical.

The following theorem summarizes.

\begin{thm}
The functor $G$ induces a quiver isomorphism between the AR-quiver of
$\mathcal{C}^m_{p,q}$ and the AR-quiver of $\mathcal{C'} \subset
\mathcal{C}^m$. Furthermore $G$ is dense and faithful, and we have
that $F(\alpha[s]) = F(\alpha)[s]$, for any $m$-diagonal $\alpha$ and
all integers $s$.
\end{thm}

%NOTES TO SELF: 
%Want more... closer to equivalence. ind C. Equivalence of subcategories.
%Mesh relations.
%Can I define the maps in the infinite radical?
%Is the category of diagonals triangulated? prove the axioms? 

%\section{Tilting objects in the category of $m$-diagonals}

%In this section we will indicate a proof that $(m+2)$-angulations
%correspond to $m$-cluster tilting objects in the $m$-cluster
%category. In other words, $G(\Delta)$ is an $m$-cluster tilting
%object. Since $G$ is dense, this gives a 1-1 correspondence. 

%We know that any $m$-cluster tilting object $T$ has a projective
%summand up to shift, and that $\End(T) \simeq \End(T[i])$ for all
%integers $i$.

%Now, let $\Delta$ be some $(m+2)$-angulation, and let $\alpha$ be an
%$m$-diagonal of Type 1. 

\begin{comment}
\section{AR-triangles and triangulated categories}
\section{Extensions and morphisms of non-rigid objects}
\section{Morphisms in the infinite radical}
\end{comment}

\section{Bijection between the mutation class and an equivalence class
of triangulations}\label{bijection}

Let $\mathcal{T}_{p,q,m}$ be the set of all $(m+2)$-angulations of
$P_{p,q,m}$, and let $\mathcal{M}_{p,q,m}$ be the mutation class of
$m$-coloured quivers of type $\widetilde{A}_{p,q}$. Recall that we
have surjective function $$\sigma_{p,q,m}: \mathcal{T}_{p,q,m}
\rightarrow \mathcal{M}_{p,q,m},$$ where $\sigma_{p,q,m}(\Delta) =
Q_{\Delta}$. As we have already pointed out, this function commutes with
mutation, i.e. $\sigma_{p,q,m}(\mu_i \Delta) =
\mu_i(\sigma_{p,q,m}(\Delta)) = \mu_i Q_{\Delta}$. We also know that
$\sigma_{p,q,m}$ is not injective, since for example
$\sigma_{p,q,m}(\Delta) = \sigma_{p,q,m}(\tau \Delta)$.   

We want to find an equivalence relation $\sim$ on
$\mathcal{T}_{p,q,m}$ such that we obtain a bijection induced from
$\sigma_{p,q,m}$, $$\widetilde{\sigma}_{p,q,m}:
(\mathcal{T}_{p,q,m}/\!\!\sim) \rightarrow \mathcal{M}_{p,q,m}.$$  

We define two functions, $r_O$ and $r_I$, on the set of
diagonals. Let $\alpha$ be a diagonal. Define $r_O \alpha$ to
be the diagonal obtained from $\alpha$ by rotating the outer
polygon $1$ step clockwise and $r_I \alpha$ to be the diagonal
obtained by rotating the inner polygon $1$ step
counterclockwise. This function is not well-defined on $m$-diagonals
when $m>1$, for if $\alpha$ is an $m$-diagonal, then $r_O \alpha$ and 
$r_I \alpha$ are not necessarily $m$-diagonals when $\alpha$ is of
Type 1. However, we note that the functions $r_O r_I = r_I r_O$,
$r_O^m$ and $r_I^m$ are well-defined on $m$-diagonals, and also that
$r_O r_I = [1]$ and consequently $r_O^m r_I^m = \tau$. We let
$r_O^{-1}$ and $r_I^{-1}$ be rotating in the opposite direction.   

For simplicity, we will in this section consider sets of diagonals
that divides $P_{p,q,m}$ into $(m+2)$-gons. In other words, we will
allow diagoals that do not necessarily satisfy the restrictions for 
$m$-diagonals. This is just for simplicity, and if we can find the
desired equivalence relation on this set, it induces an equivalence
relation on the set of $(m+2)$-angulations consisting of only
$m$-diagonals. We will still call elements in this bigger set
$(m+2)$-angulations. We can associate to an element in this set a
coloured quiver as before.

If $\Delta$ is an $(m+2)$-angulation, we define $r_O \Delta$ ($r_I
\Delta$) to be the set of diagonals obtained from $\Delta$ by
applying $r_O$ ($r_I$) on each diagonal in $\Delta$. Similarly we
define $r_I \Delta$. Note that by rotating the outer or inner polygon,
$(m+2)$-gons and hence the corresponding coloured quivers are
preserved.  

%By Lemma \ref{atleastonediagonal} we have that at least one
%diagonal in an $(m+2)$-angulation is of Type 1. So to determine the
%equivalence relation and prove the desired bijection, we can assume
%that for any $(m+2)$-angulation $\Delta$, we are in the situation
%shown in the first picture in Figure \ref{figfactoringdiagonal}, and
%hence that the winding number of any $m$-diagonal is $\leq 1$. 

We define another function on the set of $(m+2)$-angulations that
sends an $(m+2)$-angulation of $P_{p,q,m}$ to an $(m+2)$-angulation of 
$P_{q,p,m}$. Given an $(m+2)$-angulation, we can "flip" it by making
the outer polygon the inner polygon and the inner polygon the outer
polygon. We shall see that this operation corresponds to reversing
every arrow in the corresponding quiver. We can visualize this
operation as continuously stretching the $(m+2)$-angulation in three dimensions,
such that the interior together with the diagonals become the side
surface of a sylinder, the inner polygon becomes the top of the
sylinder and the outer polygon becomes the bottom. Then we push the
sylinder back into the plane, making the top of the sylinder the outer
polygon and the bottom the inner polygon, thus "flipping" the
$(m+2)$-angulation. We denote the flipped $(m+2)$-angulation $\Delta$
by $\epsilon(\Delta)$. Clearly $\epsilon(\epsilon(\Delta)) = \Delta$.

%Let $\Delta$ be an $(m+2)$-angulation of $P_{p,q,m}$, and let
%¤$\epsilon(\Delta)$ be the triangulation of $P_{q,p,m}$ obtained from
%$\Delta$ in the following way. Describe the operation!!!!!!!  Note that this operation is
%well-defined on $m$-diagonals.

\begin{lem}
If $\Delta$ is an $(m+2)$-angulation of $P_{p,q,m}$, then
$\epsilon(\Delta)$ is an $(m+2)$-angulation of $P_{q,p,m}$. The quiver
$Q_{\epsilon(\Delta)}$ is obtained from $Q_{\Delta}$ by reversing all
arrows.  
\end{lem}
\begin{proof}
It is clear that $\epsilon(\Delta)$ is an $(m+2)$-angulation of
$P_{q,p,m}$. Also, it is easy to see that any inner $(m+2)$-gon is
preserved by the flip, but that it changes orientation (it is turned
upside down). Hence the arrows in $Q_{\epsilon(\Delta)}$ are reversed. 
\end{proof}

We know that the mutation classes of $\widetilde{A}_{p,q}$ and
$\widetilde{A}_{q,p}$ are equal up to isomorphism of quivers. %To see
%this we just observe that the quiver consisting of a cycle with $p$
%arrows of colour $0$ clockwise and $q$ arrows of colour $0$
%counterclockwise, is isomorphic to the quiver consisting of a cycle
%with $q$ arrows of colour $0$ clockwise and $p$ arrows of colour $0$
%counterclockwise. 
Suppose $p \neq q$. Let $Q_{\Delta}$ be the quiver consisting of a
cycle with $p$ arrows of colour $0$ clockwise and $q$ arrows of colour
$0$ counterclockwise, i.e. we have fixed the quiver $Q_{\Delta}$ in the
plane. Then we can not reach the quiver $Q_{\epsilon(\Delta)}$ from
the quiver $Q_{\Delta}$ by a finite sequence of mutations. 
%In other words, $Q_{\Delta}$ is in the
%mutation class of $\widetilde{A}_{p,q}$ and $Q_{\epsilon(\Delta)}$ is
%in the mutation class of $\widetilde{A}_{q,p}$, if we separate these
%two classes.
%However, if $p \neq q$, the quiver $Q_{\epsilon(\Delta)}$ can not be reached
%from $Q_{\Delta}$ by a finite sequence of mutations after we fixed an
%embedding into the plane. In other words, $Q_{\Delta}$ is in the
%mutation class of $\widetilde{A}_{p,q}$ and $Q_{\epsilon(\Delta)}$ is
%in the mutation class of $\widetilde{A}_{q,p}$.

If $p=q$, we may have that flipping a triangulation $\Delta$ preserves 
the corresponding quiver up to isomorphism, i.e. $Q_{\Delta}$ may be
isomorphic to $Q_{\epsilon(\Delta)}$. This is not always the
case, but we do have that $Q_{\epsilon(\Delta)}$ can be reached from
$Q_{\Delta}$ by a finite sequence of mutations, since any orientation
of $\widetilde{A}_{p,p}$ is in the mutation class. Let $Q$ be a
coloured quiver. We say that $Q$ is reflection-symmetric if the quiver
obtained from $Q$ by reversing every arrow is isomorphic to
$Q$. Clearly, if $Q_{\Delta}$ is reflection-symmetric then
$Q_{\Delta}$ is isomorphic to $Q_{\epsilon(\Delta)}$. If $Q_{\Delta}$
is isomorphic to $Q_{\epsilon(\Delta)}$, we say that $\Delta$ is
reflection-symmetric. 
%Also, we say that an $m$-cluster tilting object
%$T$ corresponding to a reflection-symmetric $(m+2)$-angulation is
%reflection-symmetric. 

For $p \neq q$, we define an equivalence relation on
$\mathcal{T}_{p,q,m}$ by letting two $(m+2)$-angulations $\Delta$ and
$\Delta'$ be equivalent if and only if $\Delta' = r_O^i r_I^j \Delta$
for some integers $i$ and $j$. If $p=q$, we define the two
$(m+2)$-angulations to be equivalent if and only if $\Delta' = r_O^i
r_I^j \epsilon^k \Delta$ for some integers $i$, $j$ and $k \in
\{0,1\}$, where $k=0$ if $\Delta$ is not reflection-symmetric. We
write $(\mathcal{T}_{p,q,m}/\!\! \sim)$ for the class of equivalent 
$(m+2)$-angulations thus obtained. Now, clearly this gives a
map $$\widetilde{\sigma}_{p,q,m}: (\mathcal{T}_{p,q,m}/\!\!\sim)
\rightarrow \mathcal{M}_{p,q,m},$$ induced from $\sigma_{p,q,m}$, and
we claim that this map is bijective. By the discussion above we
already have surjectivity.

Recall that for an $m$-diagonal $\alpha$ in an $(m+2)$-angulation
$\Delta$, we always denote by $v_{\alpha}$ the corresponding vertex in 
$Q_{\Delta}$. If $\Delta$ is an $(m+2)$-angulation and $\alpha$ an
$m$-diagonal, we want to investigate the procedure of factoring out
the corresponding vertex $v_{\alpha}$ in $Q_{\Delta}$. We say that
$\alpha$ is close to the border of the outer polygon if $\alpha$ is an
$m$-diagonal of Type 2 and homotopic to $\delta_{i,m+2}$ for some
$i$, i.e. $\alpha$ is in an $(m+2)$-gon together with edges only on
the border. Similarly $\alpha$ is close to the border of the inner
polygon if $\alpha$ is of Type 3 and homotopic to $\gamma_{i,m+2}$ for
some $i$. Recall that $m$-diagonals close to the border corresponds to
quasi-simple objects in the $m$-cluster category, and all
$m$-diagonals of Type 2 and 3 corresponds to objects in the
tubes. Objects in the transjective components corresponds to
$m$-diagonals of Type 1. 

\begin{lem}\label{existclosetotheborder}
If $\Delta$ is an $(m+2)$-angulation containing an $m$-diagonal
$\alpha$ of Type 2 (Type 3), then there exist an $m$-diagonal close to
the border of the outer (inner) polygon. 
\end{lem}
\begin{proof}
Suppose $\alpha$ is of Type 2 and not close to the border of the outer
polygon. Then $\alpha$ divides the polygon into two parts $A$ and $B$,
where one part, say $A$, contains the inner polygon. Then $B$ is just
an $(m+2)$-angulation of a regular polygon, so there exist an
$m$-diagonal that divides $B$ into two smaller parts. By induction,
there exist an $m$-diagonal close to the border. The proof for
$m$-diagonals of Type 3 is similar. 
\end{proof}

%\begin{cor} If an $m$-cluster tilting object $T$ has an indecomposable
%  object $T_i$ of quasi-length $s$ in a tube $\mathcal{T}$ as direct
%  summand, then at least $s-1$ direct summands of $T$ are in some copy
%  (possibly different copies) of $\mathcal{T}$ and with quasi-length
%  $< s$, and at least one of them is quasi-simple. 
%\end{cor}
%\begin{proof}
%The indecomposable object $T_i$ corresponds to an $m$-diagonal
%$\alpha$ of Type 2 (3), and it divides $P_{p,q,m}$ into two parts $A$
%and $B$, where, say, $B$ contains the inner (outer) polygon. Since
%$T_i$ has quasi-length $s$, the part $B$ is an $(m+2)$-angulation of a
%polygon with $sm+2$ vertices. All $m$-diagonals in $B$ are of Type
%2 (3) and with quasi-length $<s$. An $(m+2)$-angulation of a polygon
%with $sm+2$ vertices contains exactly $s-1$ $m$-diagonals
%\cite{bm2}. One of these $m$-diagonals is close to the border of the
%outer (inner) polygon by Lemma \ref{existclosetotheborder}, and the
%corresponding indecomposable object is quasi-simple. 
%\end{proof}

If $\alpha$ is close to the border of the outer (inner) polygon, we
define $\Delta / \alpha$ to be the $(m+2)$-angulation obtained from
$\Delta$ by letting $\alpha$ be a border edge of the outer (inner)
polygon and leaving all the other $m$-diagonals unchanged. We say that
we factor out $\alpha$. See Figure \ref{figfactoringouterinner}.

  \begin{figure}[htp]
  \begin{center}
    \includegraphics[width=6.7cm]{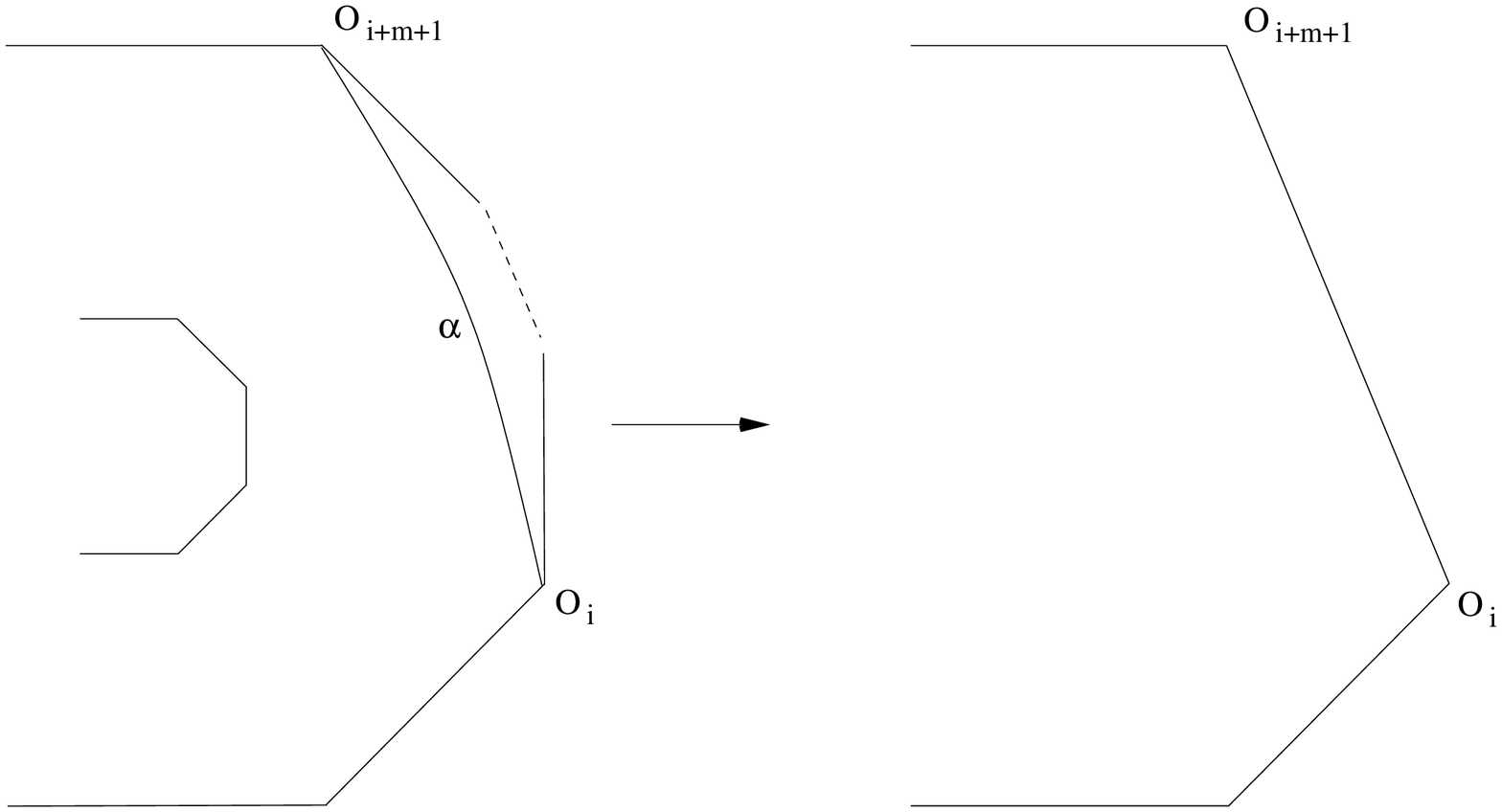}
    \includegraphics[width=6.7cm]{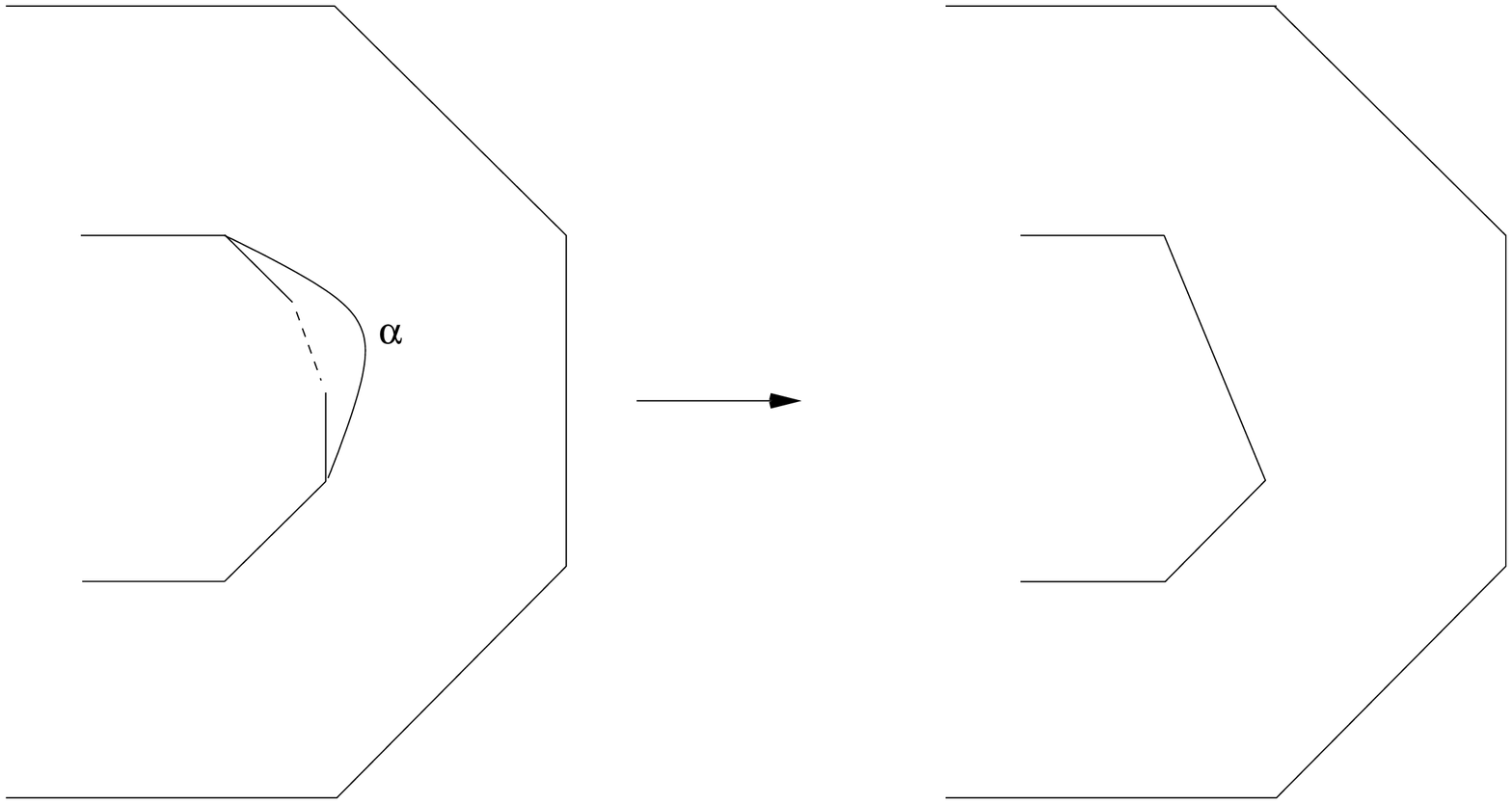}
  \end{center}\caption{\label{figfactoringouterinner} Factoring out a
    diagonal $\alpha$ close to the border of the outer and inner
    polygon, in the case $m=1$.}    
  \end{figure}

\begin{lem}
Let $\Delta$ be an $(m+2)$-angulation of $P_{p,q,m}$ and let $\alpha$
be close to the border of the outer (inner) polygon. Then the quiver
$Q_{\Delta}/v_{\alpha}$, obtained from $Q_{\Delta}$ by factoring out the
vertex $v_{\alpha}$ is connected and of type $\widetilde{A}_{p-1,q}$
($\widetilde{A}_{p,q-1}$). Furthermore, we have that
$Q_{\Delta/\alpha} = Q_{\Delta}/v_{\alpha}$.
\end{lem}
\begin{proof}
We refer to \cite{to3} and the proof there for Dynkin type $A$. The
proof in this case is a straightforward adaption.
\end{proof}

Next we consider factoring out vertices that correspond to
$m$-diagonals not close to the border.

\begin{lem}
Let $\Delta$ be an $(m+2)$-angulation. If we factor out a vertex in
$Q_{\Delta}$ corresponding to an $m$-diagonal not close to the border
and not of Type 1 (an $m$-diagonal between the outer and inner
polygon), then the resulting quiver is disconnected.
\end{lem}
\begin{proof}
In this case $\alpha$ divides the polygon into two parts $A$ and $B$,
where one part, say $A$, contains the inner polygon. Let $\beta$ be an
$m$-diagonal in $A$ and $\beta'$ an $m$-diagonal in $B$. If $\beta$ and
$\beta'$ would determine a common $(m+2)$-gon, the third $m$-diagonal
would have to cross $\alpha$, hence there is no arrow between the
subquiver determined by $A$ and the subquiver determined by $B$,
except those passing through $v_{\alpha}$. Thus factoring out
$v_{\alpha}$ disconnects the quiver.
\end{proof}

Let $\Delta$ be an $(m+2)$-angulation in $\mathcal{T}_{p,q,m}$. We
know by Lemma \ref{atleastonediagonal} that $\Delta$ contains at least
one $m$-diagonal $\alpha$ of Type 1. We want to define a function on
$\Delta$, which correspond to factoring out the corresponding vertex
$v_{\alpha}$ in the quiver $Q_{\Delta}$, when $\alpha$ is of Type
1. By rotating the outer and inner polygon we can assume that we are
in the situation shown in the first picture in Figure
\ref{figvisualizingfac}, and hence that the winding number of any
diagonal in $\Delta$ is $\leq 1$. We cut the $(m+2)$-angulation along
$\alpha$ as shown in Figure \ref{figvisualizingfac}. We obtain two new
border edges $\alpha'$ and $\alpha''$ in a regular polygon. All the
other diagonals are left unchanged.

\begin{comment}
First remove the diagonal $\alpha$. Then we create $q+3$ new
vertices on the outer polygon, where the vertices on the inner polygon
are identified with the new vertices as shown in Figure
\ref{figfactoringdiagonal}. Diagonals with end-point at the left of
$\alpha$ in $O_i$ will now have corresponding end-point in vertex
$a$. Similarly, diagonals with end-point to the right of $\alpha$ will
have corresponding end-point in $d$. Diagonals to the left of
$\alpha$ in $I_j$ will have end-point in $b$ and the diagonals to the
right will have corresponding end-point in $c$. The vertices $O_i$ and
$I_j$ are removed. Figure \ref{figvisualizingfac} might give a nice
way of visualizing the operation. 

  \begin{figure}[htp]
  \begin{center}
    \includegraphics[width=12.5cm]{factoringdiagonal.eps}
  \end{center}\caption{\label{figfactoringdiagonal} Factoring out an
    $m$-diagonal $\alpha$ of Type 1.}   
  \end{figure}
\end{comment}

  \begin{figure}[htp]
  \begin{center}
    \includegraphics[width=8.5cm]{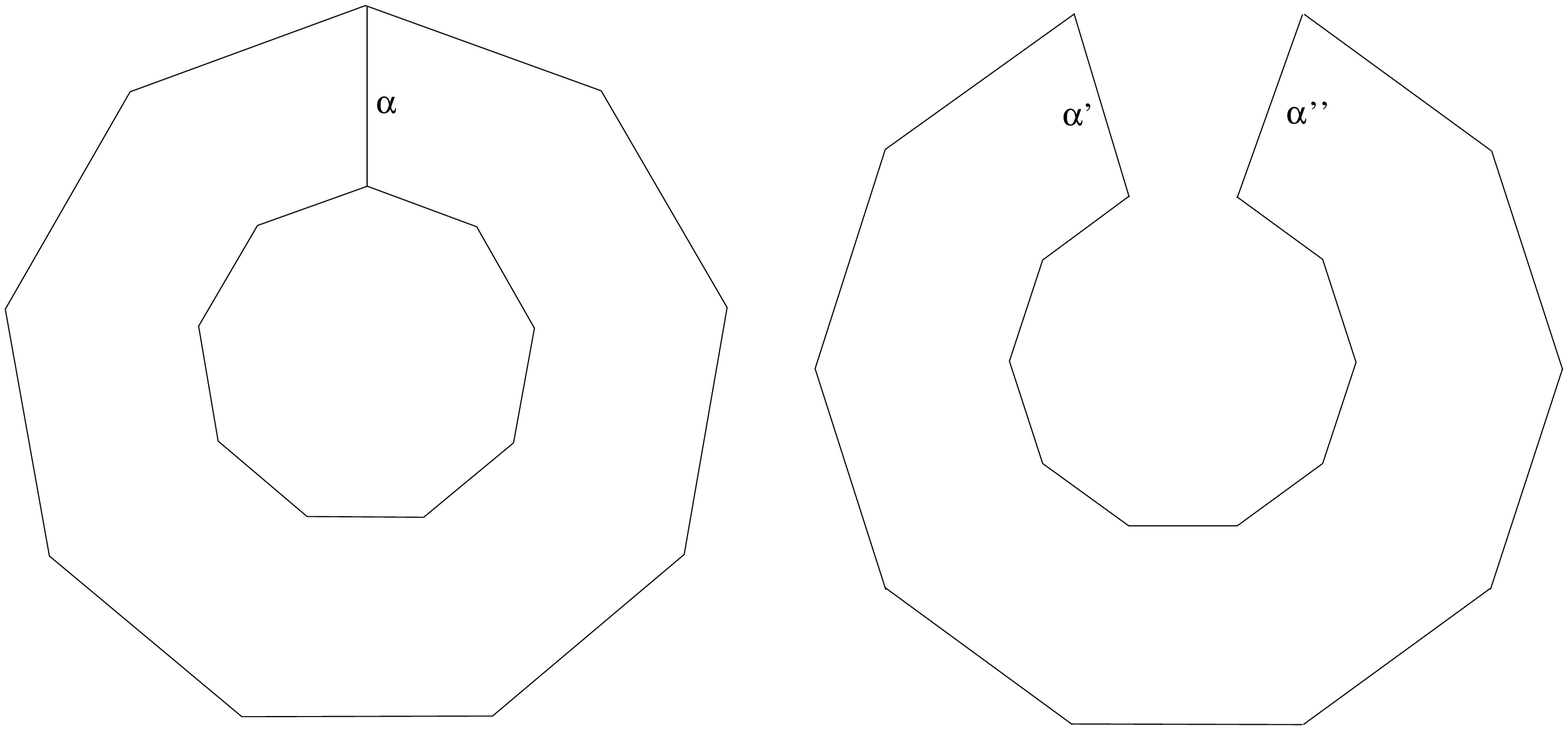}
  \end{center}\caption{\label{figvisualizingfac} Factoring out a
    diagonal $\alpha$ of Type 1.}
  \end{figure}

We have the following results.

\begin{lem}\label{factoringtype1}
Let $\Delta$ be an $(m+2)$-angulation of $P_{p,q,m}$ and let $\alpha$
be an $m$-diagonal of Type 1. Then factoring out $\alpha$ gives an
$(m+2)$-angulation of the regular polygon with $mp+mq+2$ vertices. 
\end{lem}
\begin{proof}
Clearly the resulting polygon has $mp+mq+2 = m(p+q)+2$ vertices. It is
known, for example by \cite{bm2}, that any $(m+2)$-angulation of the
regular polygon with $mn+2$ vertices has exactly $n-1$ $m$-diagonals.
There are exactly $p+q-1$ diagonals left after factoring out $\alpha$,
hence we have obtained an $(m+2)$-angulation of the regular polygon
with $m(p+q)+2$ vertices. 
\end{proof}

\begin{lem}
Let $\Delta$ be an $(m+2)$-angulation of $P_{p,q,m}$. If $\alpha$ is a 
diagonal of Type 1, then the quiver $Q_{\Delta}/v_{\alpha}$ is
connected and of Dynkin type $A_{p+q-1}$. Furthermore, factoring out
$\alpha$ corresponds to factoring out the corresponding vertex in
$Q_{\Delta}$, i.e. $Q_{\Delta / \alpha} = Q_\Delta / v_{\alpha}$.
\end{lem}
\begin{proof}
Factoring out $\alpha$ does not affect the inner $(m+2)$-gons in
$\Delta$, and so hence the arrows between vertices not equal to
$v_\alpha$ stay the same. The arrows from and to $v_\alpha$ are
removed. 
\end{proof}

We note that this procedure is reversible. Let $\Delta'$ be an
$(m+2)$-angulation of a regular polygon with $mp+mq+2$
vertices. Suppose we want an $(m+2)$-angulation of $P_{p,q}$. Pick any
border edge $e$. Then there are two possible border edges we can can
identify with $e$ to obtain an $(m+2)$-angulation of $P_{p,q}$.

Summarizing we obtain the following proposition.

\begin{prop}\label{iffdiagonals}
Let $\Delta$ be an $(m+2)$-angulation of $P_{p,q,m}$ and $\alpha$ an
$m$-diagonal. Then the $m$-coloured quiver $Q_{\Delta}/v_{\alpha}$ is
connected and of Dynkin type $A_{p+q-1}$ if and only if $\alpha$ is of
Type 1. Also $Q_{\Delta}/v_{\alpha}$ is connected and of Dynkin type
$\widetilde{A}_{p-1,q}$ ($\widetilde{A}_{p,q-1}$) if and only if
$\alpha$ is close to the border of the outer (inner) polygon. Else
$Q_{\Delta}/v_{\alpha}$ is disconnected. 
\end{prop}

%We note that this means that factoring out a vertex corresponding to a
%quasi-simple, gives a quiver of an $m$-cluster tilted algebra of type
%$\tilde{A}_{p-1,q}$ or $\tilde{A}_{p,q-1}$. Factoring out a vertex
%corresponding to an object not in a tube (i.e. shift of a
%preprojective or preinjective), gives a quiver of an $m$-cluster
%tilted algebra of type $A$, while factoring out an object in a tube
%that is not quasi-simple disconnects the quiver.

Let $\Delta$ be an $(m+2)$-angulation of $P_{p,q,m}$, and let $e$ be a
border edge between $i$ and $j$ on the outer or inner polygon. We
define another $(m+2)$-angulation $\Delta(e)$ of $P_{p+1,q,m}$ (if $e$
is on the outer polygon) or $P_{p,q+1,m}$ (if $e$ is on the inner
polygon), called the extension of $\Delta$ at $e$. The new
$(m+2)$-angulation $\Delta(e)$ is obtained from $\Delta$ by adding $m$
new border vertices between $i$ and $j$ and letting $e$ be an
$m$-diagonal close to the border in $\Delta(e)$. See Figure
\ref{figextending}.

  \begin{figure}[htp]
  \begin{center}
    \includegraphics[width=8.5cm]{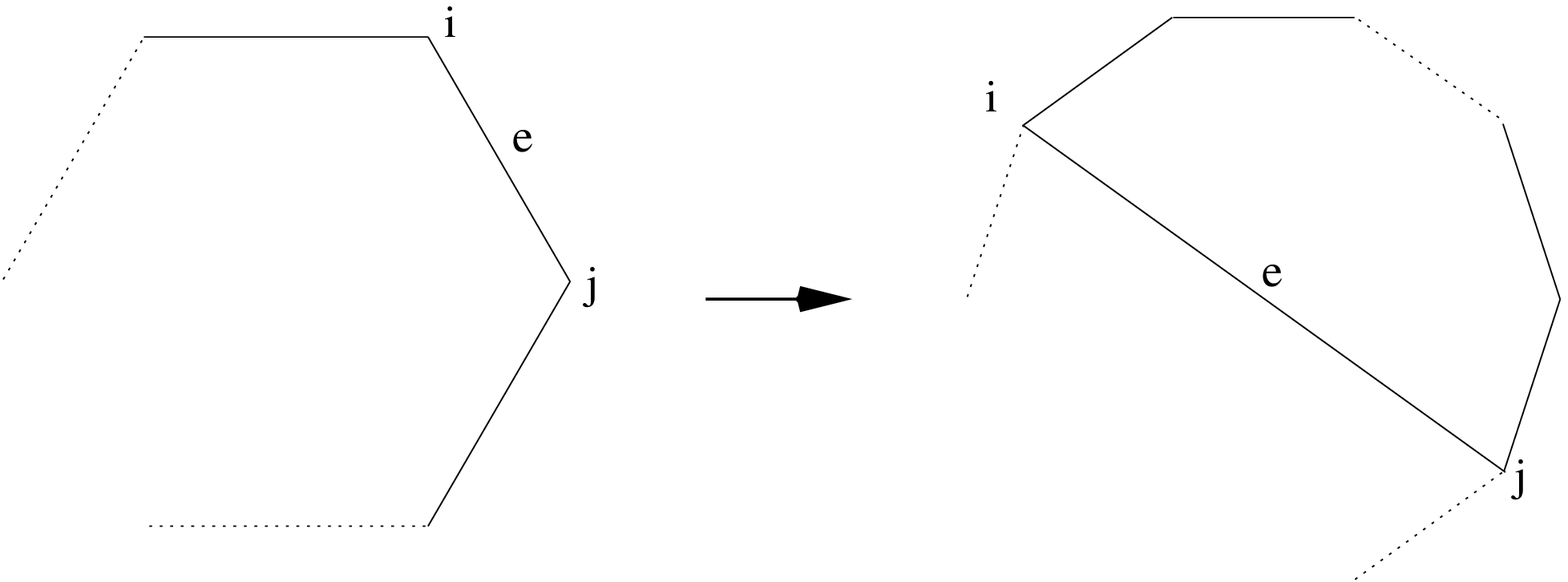}
    \includegraphics[width=8.5cm]{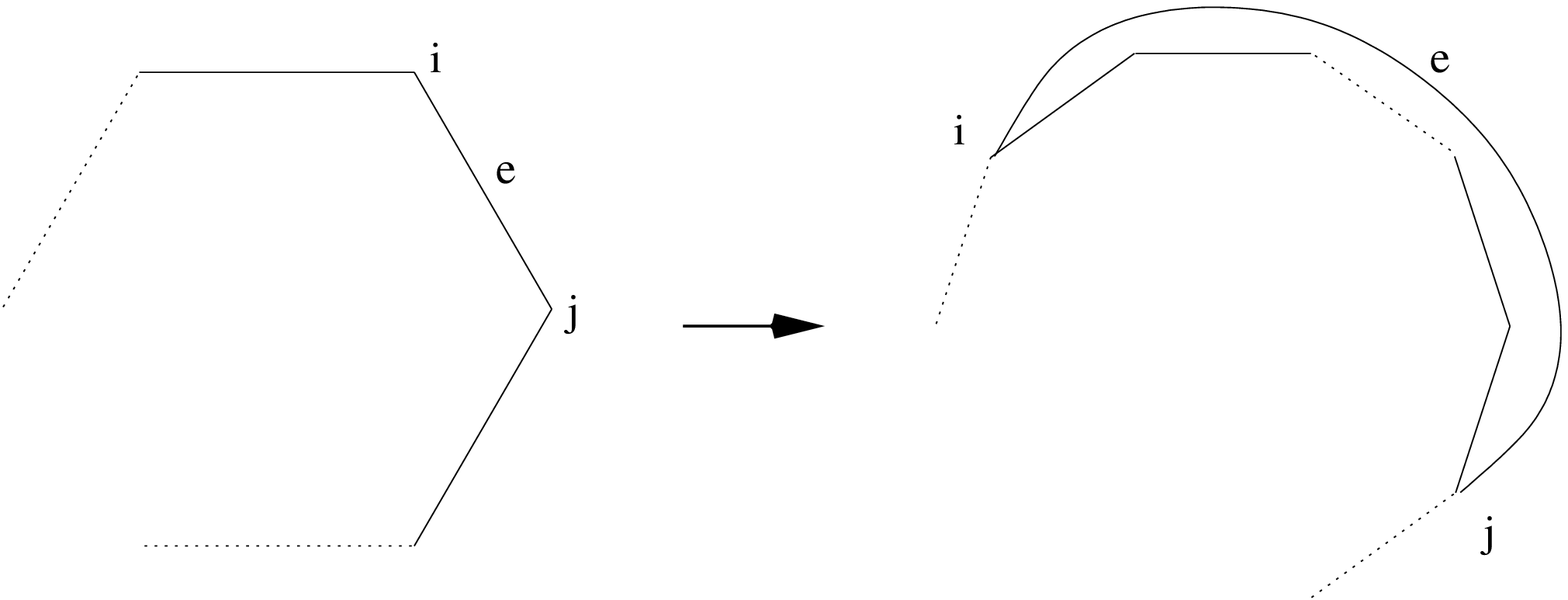}
  \end{center}\caption{\label{figextending} Extension of $\Delta$
    at the border edge $e$ on the outer and inner polygon respectively.}   
  \end{figure}

Now we are ready to prove the bijection between
$\mathcal{T}_{p,q,m}/\!\!\sim$ and $\mathcal{M}_{p,q,m}$. We look at
two different cases, namely the case when $(m+2)$-angulations only
contain $m$-diagonals of Type 1 and the case when $(m+2)$-angulations
contain at least one $m$-diagonal of Type 2 or 3.

Suppose $\Delta$ is an $(m+2)$-angulation that contains only
$m$-diagonals of Type 1, i.e. $m$-diagonals between the outer and
inner polygon. Then the corresponding coloured quiver $Q_{\Delta}$ is
a cycle of length $p+q$, with arrows of possibly different colours
going both ways. Given such a quiver $Q$, we want to show that there
is a unique $(m+2)$-angulation in $\mathcal{T}_{p,q,m}/\!\!\sim$ that
maps to $Q$. 

\begin{lem}\label{onlytype1}
Let $Q$ be a coloured quiver consisting of a cycle. If $\Delta$ and
$\Delta'$ are two $(m+2)$-angulations such that
$\sigma_{p,q,m}(\Delta) = \sigma_{p,q,m}(\Delta') = Q$, then $\Delta =
\Delta'$ in $(\mathcal{T}_{p,q,m}/\!\!\sim)$.  
\end{lem}
\begin{proof} We sketch a proof.
Start with an $m$-diagonal $\alpha$ between $I_i$ and $O_j$. Choose a
vertex $v_{\alpha}$ in $Q$ corresponding to $\alpha$. Suppose there is
an arrow $r$ from $v_{\alpha}$ to $v_{\beta}$ of colour $c$. Then $\beta$
is either a diagonal between $I_{i+m-c}$ and $O_{j+c}$ if $r$
goes counterclockwise or a diagonal between $I_{i+c}$ and $O_{j+m-c}$
if $r$ goes clockwise. By induction we can continue. If $p =
q$ and if the quiver obtained from $Q$ by reversing all arrows is
isomorphic to $Q$, then the corresponding $(m+2)$-angulations are
flips and rotations of eachother.
\end{proof}

For the case when the $(m+2)$-angulation contains $m$-diagonals of
Type 2 or 3, we need two more lemmas. 
%before we prove the main theorem
%in this section. 

\begin{lem}\label{equivrot}
Let $\Delta$ be an $(m+2)$-angulation. Suppose that $Q_{\Delta} \simeq
Q_{\Delta'}$ for some $(m+2)$-angulation $\Delta'$ implies $\Delta =
\Delta'$ in $(\mathcal{T}_{p,q,m}/\!\!\sim)$. 

Let $\alpha$ be an $m$-diagonal in $\Delta$. Suppose there is an
isomorphism $Q_{\Delta} \stackrel{\theta}{\rightarrow} Q_{\Delta'}$ that
sends $v_{\alpha}$ to $v_{\alpha'}$. Then $\alpha' = r_O^i r_I^j
\epsilon^k \alpha$ and $\Delta' = r_O^i r_I^j \epsilon^k \Delta$ for
some integers $i$, $j$ and $k \in \{0,1\}$, where $k=0$ if $Q_\Delta$
is not reflection-symmetric.
\end{lem}
\begin{proof}
If $\alpha$ is of Type 2 or 3, the proof is a straightforward adaption
of the proof in \cite{to3} for the Dynkin $A$ case, and we omit it. 

Suppose $\alpha$ is of Type 1. Then $\alpha'$ is also of Type
1 by Proposition \ref{iffdiagonals}, and there exist some $i$, $j$ and
$k$ such that $\alpha' = r_O^i r_I^j \epsilon^k \alpha$. Take any
$m$-diagonal $\beta$ in $\Delta$ such that $v_\alpha$ has an arrow of
colour $c$ to $v_\alpha$, i.e. $\alpha$ and $\beta$ are in a common
$(m+2)$-gon in $\Delta$. Suppose $\theta(v_{\beta}) =
v_{\beta'}$. Then $\beta'$ and $\alpha'$ are in a common $(m+2)$-gon
in $\Delta'$, and $v_{\alpha'}$ has an arrow of colour $c$ to
$v_{\beta'}$.   

Factoring out $\alpha$ and $\alpha'$, as described right before Lemma
\ref{factoringtype1}, gives $(m+2)$-angulations $\Delta / \alpha$ and
$\Delta' / \alpha'$ of the regular polygon with $mp+mq+2$ vertices, and
corresponding quivers of type $A_{p+q-1}$. Then, clearly, $Q_{\Delta}
/ v_{\alpha} = Q_{\Delta / \alpha} \simeq Q_{\Delta' / \alpha'} =
Q_{\Delta'} / v_{\alpha'}$, and from \cite{to3} it follows that there
exist some integer $i'$ such that $\beta' = r^{i'} \beta$ and $\Delta'
/ \alpha' = r^{i'} \Delta / \alpha$, where $r^{i'}$ is rotating the
$(m+2)$-angulation of the regular polygon $i'$ steps in the
counterclockwise direction. Since $v_\beta$ has an arrow of colour $c$
to $v_\alpha$ and $v_\beta'$ has an arrow of colour $c$ to
$v_\alpha'$, the claim follows. 
\end{proof}

\begin{lem}\label{isoextension}
Let $\Delta$ be an $(m+2)$-angulation, and suppose that $Q_{\Delta}
\simeq Q_{\Delta'}$ for some $(m+2)$-angulation $\Delta'$ implies
$\Delta = \Delta'$ in $(\mathcal{T}_{p,q,m}/\!\!\sim)$.  

Let $\Delta(e)$ and $\Delta(e')$, with $e \neq e'$, be two extensions
of $\Delta$. Then $\Delta(e) = \Delta(e')$ in
$(\mathcal{T}_{p,q,m}/\!\!\sim)$ 
%$\Delta(e) = r_O^i r_I^j \epsilon^k \Delta(e')$ for
%some integers $i$, $j$ and $k \in \{0,1\}$ 
if and only if
$Q_{\Delta(e)} \simeq Q_{\Delta(e')}$.  
\end{lem}
\begin{proof}
Suppose $Q_{\Delta(e)} \simeq Q_{\Delta(e')}$. The coloured quiver
$Q_{\Delta(e)}$ is connected, so suppose $v_e$ has an arrow of colour
$c$ to $v_{\alpha}$, where $\alpha$ is an $m$-diagonal in
$\Delta(e)$. This means that $e$ is an edge in some $(m+2)$-gon
together with $\alpha$. Then there is some $v_{\beta}$ in
$Q_{\Delta(e')}$ with an arrow of colour $c$ to some vertex
$v_{\gamma}$, such that there exist an isomorphism $Q_{\Delta(e)}
\rightarrow Q_{\Delta(e')}$ sending $v_{\alpha}$ to $v_{\gamma}$ and 
$v_e$ to $v_{\beta}$. By Proposition \ref{iffdiagonals} we have that $\beta$ is
close to the border in $\Delta(e')$, since $e$ is close to the border
in $\Delta(e)$. Then we have that
$$Q_{\Delta} \simeq Q_{\Delta(e)/e} \simeq Q_{\Delta(e)} / v_e \simeq
Q_{\Delta(e')}/v_{\beta} \simeq Q_{\Delta(e') / \beta},$$
so hence, by assumption and Lemma \ref{equivrot}, $\Delta(e') / \beta =
  r_O^i r_I^j \epsilon^k \Delta$ and $\gamma = r_O^i r_I^j \epsilon^k
  \alpha$, where $k=0$ if $\Delta$ is not reflection-symmetric.
 
Suppose the $m$-diagonal $\alpha$ is of Type 2. Then $\alpha$ divides
the polygon into two parts $A$ and $B$, where, say, $B$ contains the
inner polygon. Then also $\gamma$ divides the polygon into two parts
$A'$ and $B'$, where, say, $B'$ contains the inner polygon. If $e$
lies in $A$, then $\beta$ lies in $A'$, and if $e$ lies in $B$, then
$\beta$ lies in $B'$. There is only one way to extend $\Delta(e') /
\beta$ in $A'$ (or $B'$) such that the new vertex has an arrow of
colour $c$ to $v_{\gamma}$. Hence $\beta = r_O^i r_I^j \epsilon^k e$,
and $r_O^i r_I^j \epsilon^k \Delta(e) = \Delta(e')$.

If $\alpha$ is of Type 3, we do similarly.

Suppose the $m$-diagonal $\alpha$ is of Type 1. Then $\gamma$ is also
of Type 1. There is only one way to extend $\Delta(e') / \beta$ (on
the outer or inner polygon) such that the new vertex has an arrow of
colour $c$ to $v_{\gamma}$, and as above we are done. 
%... only one way to
%extend with a certain colour and oriantation (on each side of the
%diagonal on the outer polygon.
\end{proof}

Now we can prove the main theorem in this section.

\begin{thm}\label{bijectionmain}
The function $\widetilde{\sigma}_{p,q,m}:
(\mathcal{T}_{p,q,m}/\!\!\sim) \rightarrow \mathcal{M}_{p,q,m}$ is a
bijection. 
\end{thm}
\begin{proof}
We only need to show injection. Suppose
$\widetilde{\sigma}_{p,q,m}(\Delta) =
\widetilde{\sigma}_{p,q,m}(\Delta')$. We want to show that $\Delta =
\Delta'$ in $(\mathcal{T}_{p,q,m}/\!\!\sim)$. If $\Delta$ does not
contain a diagonal close to the border, we are finished by Lemma
\ref{onlytype1} and \ref{existclosetotheborder}, so we can assume that 
there exist a diagonal $\alpha$ close to the border of the inner or
outer polygon. If $\alpha$ is close to the border of the inner
polygon, we can consider $\epsilon(\Delta)$ instead, hence we can also
assume that $\alpha$ is close to the border of the outer polygon. It
is straightforward to verify that $\widetilde{\sigma}_{2,2,m}:
(\mathcal{T}_{2,2,m}/\!\!\sim) \rightarrow \mathcal{M}_{2,2,m}$ is 
bijective. 

Fix $q \geq 2$ and let $p > 2$. Suppose that
$\widetilde{\sigma}_{p-1,q,m}: (\mathcal{T}_{p-1,q,m}/\!\!\sim)
\rightarrow \mathcal{M}_{p-1,q,m}$ is injective for all $m$.

Let $\alpha$ be close to the border in $\Delta$. Then the $m$-diagonal
$\alpha'$ in $\Delta'$ corresponding to $v_{\alpha}$ in $Q$ is also
close to the border. By hypothesis $\Delta / \alpha = \Delta' /
\alpha'$ in $\mathcal{T}_{p-1,q,m}/\!\!\sim$. We can obtain $\Delta$
and $\Delta'$ from $\Delta/\alpha = \Delta'/\alpha'$ by extension. By
Lemma \ref{isoextension} all possible extensions of $\Delta / \alpha$
and $\Delta' / \alpha'$ give non-isomorphic quivers, unless $\Delta =
\Delta'$ in $\mathcal{T}_{p,q,m}/\!\!\sim$. We do the same induction
step on $q$, and as above we are done.
\end{proof}

We obtain the following corollary.

\begin{cor}
The number $\widetilde{a}(p,q,m)$ of elements in the mutation class of
any $m$-coloured quiver of type $\widetilde{A}_{p,q}$ is equal to the
number of $(m+2)$-angulations of $P_{p,q,m}$, up to rotation of the
outer and inner polygon, and up to "flip" if $p=q$ and the
$(m+2)$-angulation is reflection-symmetric. 
\end{cor}

We mention that these numbers have already been determined in
\cite{bprs} for $m=1$, and from those results we get the following corollary.

\begin{cor}
The number of triangulations of $P_{p,q}$, up to rotation and flip of
reflection symmetric triangulations, is given by

$$\tilde{a}_{p,q} = \left\{
      \begin{array}{ll}
		\frac{1}{2} \displaystyle\sum_{k|p,k|s}
                \frac{\phi(k)}{p+q} \binom{2p/k}{p/k} \binom{2q/k}{q/k}  & \mbox{if } p \neq q \\
		\frac{1}{2} \left(\frac{1}{2} \binom{2p}{p}) +
                \displaystyle\sum_{k|p} \frac{\phi(k)}{4r} \binom{2p/k}{p/k}^2\right) & \mbox{if } p = q,
	\end{array}\right.
 $$
where $\phi(k)$ is the Euler function.

\end{cor}

Furthermore, we can also consider the number of triangulations of
the annulus with $p+q$ marked points, i.e. two triangulations are
equivalent if and only if they are rotations of eachother.

\begin{cor}
The number of triangulations of $P_{p,q}$, up to rotation, is given by
$$\frac{1}{2} \displaystyle\sum_{k|p,k|s}
                \frac{\phi(k)}{p+q} \binom{2p/k}{p/k} \binom{2q/k}{q/k},
 $$
where $\phi(k)$ is the Euler function.
\end{cor}


\small
\begin{thebibliography}{99}



\bibitem[BM1]{bm1} Baur K., Marsh R. J. \emph{A Geometric Description
    of the $m$-cluster categories of type $D_n$}, International
  Mathematics Research Notices (2007) Vol. 2007 : article ID rnm011. 
  
\bibitem[BM2]{bm2} Baur K., Marsh R. J. \emph{A geometric description
    of m-cluster categories}, Trans. Amer. Math. Soc. 360 (2008),
  5789-5803. 

\bibitem[BMRRT]{bmrrt} Buan A., Marsh R., Reineke M., Reiten I., Todorov G. 
\emph{Tilting theory and cluster combinatorics}, Advances in
mathematics, 204 (2), 572-618 (2006).

\bibitem[BMR1]{bmr1} Buan A., Marsh R., Reiten I. \emph{Cluster-tilted
    algebras}, Trans. Amer. Math. Soc., 359, no. 1, 323--332 (2007).

\bibitem[BMR2]{bmr2} Buan A., Marsh R., Reiten I. \emph{Cluster
    mutation via quiver representations}, Commentarii Mathematici
  Helvetici, Volume 83 no.1, 143-177 (2008). 

\bibitem[BMR3]{bmr3} Buan A. B., Marsh R., Reiten I. \emph{Cluster-tilted 
algebras of finite representation type}, Journal of Algebra, Volume
306, Issue 2, 412-431 (2006).

\bibitem[BPRS]{bprs} Bastian J., Prellberg T., Rubey M., Stump
  C. \emph{Counting the number of elements in the mutation classes of
    $\widetilde{A}_n$-quivers}, Electronic Journal of Combinatorics, 18 (2011) P98. 

\bibitem[BR]{br} Buan A. B., Reiten I. \emph{Acyclic quivers of finite
    mutation type}, International Mathematics Research Notices,
  Article ID 12804 (2006), 1-10.

\bibitem[BT]{bt} Buan A. B., Thomas H. \emph{Coloured quiver mutation
    for higher cluster categories}, Adv. Math. Volume 222 (3) 971-995
  (2009).

\bibitem[BTo]{bto} Buan A. B., Torkildsen H. A. \emph{The number of
    elements in the mutation class of a quiver of type $D_n$}, The
  Electronic Journal of Combinatorics, 16(1) (2009) R49.  

\bibitem[BZ]{bz} Br\"ustle T., Zhang J. \emph{On the cluster category of a
  marked surface}, to appear in Algebra and Number Theory.

\bibitem[CCS]{ccs} Caldero P., Chapoton F., Schiffler
  R. \emph{Quivers with relations arising from clusters ($A_n$ case)},
  Trans. Amer. Math. Soc. 358 , no. 3, 1347-1364 (2006). 

\bibitem[CCS2]{ccs2} Caldero P., Chapoton F., Schiffler
  R. \emph{Quivers with relations and cluster tilted algebras},
  Algebras and Representation Theory 9, no. 4, 359-376 (2006).

\bibitem[FST]{fst} Fomin S., Shapiro M., Thurston D. \emph{Cluster
    algebras and triangulated surfaces. Part I: Cluster complexes},
  Acta Math. 201, 83-146 (2008). 

\bibitem[FZ1]{fz1} Fomin S., Zelevinsky A. \emph{Cluster algebras
    I: Foundations}, J. Amer. Math. Soc. 15, 497-529 (2002).
 
\bibitem[FZ2]{fz2} Fomin S., Zelevinsky A. \emph{$Y$-systems and
    generalized associahedra}, Ann. of Math., 158 (3) (2003),
  977-1018. 

\bibitem[FZ3]{fz3} Fomin S., Zelevinsky A. \emph{Cluster algebras II:
    Finite type classification}, Inventiones Mathematicae 154, 63-121 (2003).

\bibitem[IY]{iy} Iyama O., Yoshino Y. \emph{Mutation in triangulated
    categories and rigid Cohen-Macaulay modules}, Invent. Math. 172
  (1) (2008), 117-168. 

\bibitem[K]{k} Keller B. \emph{On triangulated orbit categories},
  Doc. Math. 10 (2005), 551-581. 

\bibitem[S]{s} Schiffler R. \emph{A geometric model for cluster
  categories of type $D_n$}, J. Alg. Comb. 27, no. 1, 1-21 (2008). 

\bibitem[St]{st} Stanley R. P. \emph{Enumerative Combinatorics, volume
  2}, Cambridge studies in Advanced Mathematics 62, Cambridge
University Press (1999). 

\bibitem[T]{t} Thomas H. \emph{Defining an $m$-cluster category},
  J. Algebra 318 (1) (2007), 37-46. 

\bibitem[To1]{to1} Torkildsen H. A. \emph{Counting cluster-tilted algebras
  of type $A_n$}, International Electronic Journal of Algebra, no. 4,
149-158 (2008).

\bibitem[To2]{to2} Torkildsen H. A. \emph{Finite mutation
    classes of coloured quivers},  published in Colloq. Math. 122
  (2011), 53-58. 

\bibitem[To3]{to3} Torkildsen H. A. \emph{Coloured quivers of type $A_n$
  and the cell-growth problem},  to appear in Journal of Algebra and
Its Applications. DOI No: 10.1142/S0219498812501332. 

\bibitem[To4]{to4} Torkildsen H. A. \emph{Enumeration of mutation
    classes}, PhD thesis NTNU, December 2010.

\bibitem[W]{w} Wraalsen A. \emph{Rigid objects in higher cluster
    categories}, J. Algebra 321 (2) (2009), 532-547. 

\bibitem[Z]{z} Zhu B. \emph{Generalized cluster complexes via quiver
    representations}, J. Algebraic Combin. 27 (2008), 25-54. 

\bibitem[ZZ]{zz} Zhou Y., Zhu B. \emph{Cluster combinatorics of $d$-
    cluster categories}, J. Algebra 321 (10) (2009), 2898-2915.

\end{thebibliography}
\normalsize
\end{document}